\documentclass[10pt,reqno]{amsart}
\usepackage{enumerate}

\usepackage{amsmath}%
\usepackage{amssymb}%
\usepackage[latin1]{inputenc}
\usepackage{euscript}

     \makeatletter
     \def\section{\@startsection{section}{1}%
     \z@{.7\linespacing\@plus\linespacing}{.5\linespacing}%
     {\bfseries
     \centering
     }}
     \def\@secnumfont{\bfseries}
     \makeatother
\newtheorem{theorem}{Theorem}[section]
\newtheorem{lemma}[theorem]{Lemma}
\newtheorem{proposition}[theorem]{Proposition}

\theoremstyle{definition}
\newtheorem{definition}[theorem]{Definition}

\theoremstyle{remark}
\newtheorem{remark}[theorem]{Remark}
\numberwithin{equation}{section} 

\newcommand{\R}{\mathbb{R}}

\newcommand{\md}{\operatorname{d}}

\begin{document}




\title[Symmetries of SDEs using Girsanov transformations]{ Symmetries of Stochastic Differential Equations using Girsanov transformations}


\author[F.C. De Vecchi]{\bf Francesco C. De Vecchi}
\address{Institute for Applied Mathematics and Hausdorff Center for Mathematics, Rheinische Friedrich-Wilhelms-Universit\"at Bonn.
Endenicher Allee 60, 53115 Bonn, Germany.}
\email{ francesco.deveccchi@uni-bonn.de}

\author[P. Morando]{\bf Paola Morando}
\address{DISAA, Universit\`a di Milano. Via
Celoria, 2, Milano, Italy.}
\email{paola.morando@unimi.it}

\author[S. Ugolini]{\bf Stefania Ugolini}
\address{Dipartimento di Matematica, Universit\`a di Milano. Via
Saldini 50, Milano, Italy.}
\email{ stefania.ugolini@unimi.it}

\begin{abstract}
Aiming at enlarging the class of symmetries of an SDE,  we introduce a family of stochastic transformations able to change also the underlying probability measure exploiting Girsanov Theorem and we provide new determining equations for the  infinitesimal symmetries of the SDE. The well-defined subset of the previous class of measure transformations given by Doob transformations allows us to recover all the Lie point symmetries of the Kolmogorov equation associated with the SDE. This gives the first stochastic interpretation of all the deterministic symmetries of the Kolmogorov equation. The general theory is applied to some relevant stochastic models.
\end{abstract}


\keywords{Stochastic Differential Equations, Lie symmetry analysis, Kolmogorov Equation, Girsanov transformations, Doob h-transforms}

\maketitle

\section{Introduction}

The analysis of infinitesimal symmetries of (deterministic) ordinary and partial differential equations (ODEs and PDEs) is a classical and well-developed tool in studying differential equations providing a very powerful method to both compute some of their explicit solutions or to analyze their qualitative behavior (see \cite{Bluman,Olver}  for an introduction to the subject).\\
The application of Lie symmetry analysis techniques to stochastic differential equations (SDEs) and stochastic partial differential equations (SPDEs) has received a growing interest in recent years. Without claiming to be exhaustive, from the first studies on  Brownian-motion-driven SDEs (see \cite{Co1991,AlFei1995,GaQui1999}) and Markov processes (see \cite{Glo1991,Glo1990,Liao1992,DeLara1995}) many different notions of symmetries for Brownian-motion-driven SDEs has been proposed (see the symmetries with random time change \cite{Unal2003,Srihirun2006,Fredericks2007}, W-symmetry in \cite{Ga2019,Ga2000}, weak symmetries in  \cite{DMU}, random symmetries in \cite{GaLui2017,GaSpa2017, Kozlov2018} and symmetries of diffusions in \cite{DeMoUg2019}).  Moreover, random transformations of semimartingales on Lie groups has been introduced (see \cite{AlDeMoUg2018}) and the notion of symmetry has been  extended also to SDEs driven by general semimartingales (see \cite{AlDeMoUg2019,LaOr2009}).
Furthermore the idea of reduction and reconstruction of differential equations admitting a solvable Lie algebra of symmetries has been generalized from the deterministic to the stochastic setting (see \cite{LaOr2009,Ko2010} for strong symmetries and \cite{DMU2} for weak symmetries).
Symmetries of SDEs arising from variational problems has been considered and Noether theorem and integration by quadratures have been generalized  to variational stochastic problems (see \cite{LaOr2008,PriZa2010,ThiZa1997}).
Finally,  a  generalization of differential constraint method to SPDEs  has been proposed in  \cite{De2017,DeMo2016}, while in \cite{AlDeMoUg2019,DeUg2017} symmetries have been applied to  the study  of invariant numerical methods for SDEs.\\

In the case of Brownian-motion-driven SDEs or L\'evy-process-driven SDEs, the solution process to the equation is a Markov process. This means that the mean value of a Markovian function of the process solves a deterministic linear PDE called Kolmogorov equation (whose adjoint is the also very well known Fokker-Planck equation). In this setting, a natural problem is the comparison between the set of Lie point symmetries of Kolmogorov (or Fokker-Planck) equation and the set of symmetries of the corresponding SDE (for the importance of Lie point symmetries of Kolmogorov equation in integrating the associated diffusion process see \cite{Craddock2004,Craddock2007,Craddock2009,Gungor20181,Gungor20182}). In the case of Brownian-motion-driven SDEs all the different notions of symmetry introduced for SDEs (strong symmetries, W symmetries, weak symmetries) have been proved to be also symmetries of the corresponding Kolmogorov equation, but the converse is false (see \cite{DeMoUg2019,Ga2019,GaQui1999,Ko2012,Ko2019}). In this paper we provide an appropriate extension of the notion of symmetry of an SDE so that the converse of the previous result holds.  \\

In order to do this we extend the approach first proposed in \cite{DMU}, where the concept of weak symmetry of an SDE has been introduced considering  a three components stochastic transformation for the weak solutions to a Brownian-motion-driven SDE. Indeed, this stochastic transformation is composed by a deterministic transformation of the state variable $X$, a random time change and a random rotation of the Brownian motion $W$. Since the set of these stochastic transformations forms an infinite dimensional Lie group, it is possible to consider the associated Lie algebra, whose elements are called infinitesimal stochastic transformations. The infinitesimal stochastic transformations constitute the infinitesimal generators of the finite stochastic transformations.
When the  process $P_T(X,W)$, obtained transforming a weak solution $(X,W)$ to an SDE through the stochastic transformation $T$, is again a (weak) solution to the same SDE, we say that $T$ is a (finite) weak symmetry for the SDE. The same notion can be extended to the case of infinitesimal transformations, and this allows us to obtain explicitly a set of determining equations characterizing the infinitesimal symmetries of an SDE.\\

In the present paper we enlarge the previous class of stochastic transformations introducing a new component which represents the possibility of changing the underlying  probability measure of the solution process $(X,W)$ via the use of Girsanov theorem. We call this kind of transformations \emph{extended stochastic transformations} and the corresponding symmetries \emph{extended (weak) symmetries}.\\
The first result of the paper is, thus, the introduction of this new family of stochastic  transformations and the study of their geometric properties (they form an infinite dimensional Lie group with associated Lie algebra). \\
The second main result is that for any infinitesimal Lie point symmetry of Kolmogorov equation there exists an infinitesimal extended (weak) symmetry of the corresponding SDE. Furthermore, the family of infinitesimal extended weak symmetries of the SDE is larger than the set of Lie point symmetries of the corresponding Kolmogorov equation. Indeed, we find that any extended weak symmetry of the SDE that is also a symmetry of the Kolmogorov equation is associated with a \emph{Doob change of measure}. Since Doob transformations form a strict subset of the class of absolutely continuous changes of measure which can be obtained within a Girsanov approach,  we prove that the extended weak symmetries of an SDE are in general more than the  Lie point symmetries of the Kolmogorov equation. To the best of our knowledge, these results  represent the first purely stochastic interpretation of the whole set of symmetries of Kolmogorov equations.
In particular, for one-dimensional Brownian motion,  our stochastic symmetry analysis permits us to individuate not only  all the symmetries of the one-dimensional heat equation, but also a new infinite dimensional family of symmetries associated with the random time invariance of the process (see Remark 5.1). Analogously, for two-dimensional Brownian motion, besides the symmetries of the corresponding two-dimensional heat equation, we find two new infinite dimensional families of symmetries: the first one is related to  random time invariance  and the second one is related to  random rotations invariance of the process (see Remark 5.2).\\

The plan of the paper is the following: in Section \ref{section_1} we briefly recall the symmetry analysis of the Brownian-motion-driven SDEs given in \cite{DMU}. In Section \ref{section_2} we introduce the fourth transformation of an SDE which allows the change of the underlying probability measure into another equivalent probability measure and we show how the new global stochastic transformation acts on the solution process and on the SDE coefficients. We also provide a geometrical description of the group of stochastic transformations and we derive the new determining equations for the infinitesimal symmetries. In Section \ref{section_3} we compare the extended weak symmetries of an SDE with the  Lie point symmetries of the corresponding Kolmogorov equation, proving that the two notions coincide as long as we consider extend weak symmetries with a Doob change of measure. In Section \ref{section_4} the theoretical results of the paper are applied to a list of  SDEs which are  important both from a theoretical point of view and for their  applications.

\section{A three components stochastic transformation of an SDE \label{section_1}}

Since in this paper we generalize the stochastic transformations introduced in \cite{DMU} and \cite{DMU2} for Brownian-motion-driven SDEs, for the convenience of the reader we collect in this section
the main facts and definitions that can be useful in the following. A more detailed exposition, as well as all the proofs, can be found in  \cite{DMU} and \cite{DMU2}. In the following, the Einstein summation convention on repeated indices is used.\\

Let $M, M'$ be  open subsets of $\mathbb{R}^n$. We consider the filtered probability space $\left(\Omega,\mathcal{F}, \mathcal{F}_t ,\mathbb{P}\right)$, where $\mathcal{F}_t \subset \mathcal{F}$ is a filtration. We fix a finite time horizon $[0,\mathcal{T}]$ with $\mathcal{T}\geq0$.  Let $W=\left(W^1,\dots,W^m\right)=\left(W^\alpha \right)$ be an  $\mathcal{F}_t$-$m$-dimensional Brownian motion and let $\mu : M \to \mathbb{R}^n$ and $\sigma : M \to \text{Mat}\left(n,m\right)$ be two smooth functions.
\begin{definition}
The process $(X,W)$ solves (in a weak sense) the SDE with coefficients $\mu,\sigma$ (shortly solves the SDE $\left(\mu,\sigma\right)$) if, for all $t \in [0,\mathcal{T}]$,
\begin{equation*}
X_{t}^i -X_0^i=\int_{0}^{t  } \mu ^i \left(X_s\right)ds + \int_{0}^{t }\sigma _\alpha ^i \left(X_s\right)dW_s^\alpha \quad i=1,\dots, n.
\end{equation*}
When $\left(X,W\right)$ solves the SDE $\left(\mu,\sigma\right)$ we write, as usual
\begin{equation}\label{eqSDE}
\begin{split}
dX_t &= \mu\left(X_t\right)dt+\sigma\left(X_t\right)\cdot dW_t\\
&=\mu dt+ \sigma \cdot d W_t
\end{split}
\end{equation}
\end{definition}
In order to ensure that the integrals in the above definition are well defined we suppose that the processes $(|\mu\left(X_s\right)|^{\frac{1}{2}})_{s\in[0,\mathcal{T}]}$ and $\left(\sigma\left(X_s\right)\right)_{s\in[0,\mathcal{T}]}$ belong to the class $M^2_{loc}([0,\mathcal{T}])$, i.e. to the class of processes  $\left(Y_s\right)_{s\in[0,\mathcal{T}]}$ that are progressively measurable  and such that $\int_{0}^{t}Y_s^2\left(\omega \right)ds < \infty \text{ for almost every   } \omega \in \Omega$ and $t\in[0,\mathcal{T}]$.
We recall that, given an SDE $(\mu, \sigma)$, we can introduce the corresponding infinitesimal generator given by
\begin{equation}\label{eqGENERATORL}
L= \frac{1}{2}\left(\sigma \sigma^T\right)^{ij} \partial_i \partial_j+ \mu^i \partial_i.
\end{equation}
 In the following we only consider autonomous SDEs, and we introduce three different transformations for
 their solution processes.

\subsection*{Spatial transformations }

We call spatial transformation of an SDE a  transformation $\Phi$ acting on the process component $X$ through a generic diffeomorphism  $\Phi:M \to M'$.
Let us denote with $\nabla \Phi : M \to \text{Mat}\left(n,m\right)$ the Jacobian matrix
\begin{equation*}
	 (\nabla \Phi)_j^i=\partial_j \Phi^i.
\end{equation*}

By using the well-known It\^o formula (see, e.g., \cite{RoWi2000} Section 32 or \cite{Oksendal} Chapter 4) it is easy to compute how the coefficients of the SDE change.
\begin{proposition}\label{Prop1}
	Given a diffeomorphism $\Phi : M \to M'$, if the process $\left(X,W\right)$ is solution to the  SDE $\left(\mu,\sigma\right)$, then the process $\left(\Phi\left(X\right),W\right)$ is solution to the SDE $\left(\mu',\sigma'\right)$ with
\begin{equation*}
\begin{split}
\mu'&= L\left(\Phi\right)\circ \Phi^{-1}\\
\sigma'&=\left(\nabla \Phi \cdot \sigma\right)\circ\Phi^{-1}
\end{split}
\end{equation*}
\end{proposition}
We remark that the spatial transformation does not act on the Brownian motion.

\subsection*{Random time changes }

In the following we consider only Markovian absolutely continuous random time changes of the solution process to an SDE through a smooth and strictly positive density
 $\eta : M \to \mathbb{R}_+$. We denote by $H_\eta $ the transformation
$$ t'=\int_0^t \eta_s(X_s)ds.$$
This time transformation necessarily acts on both components of the solution process $(X,W)$.
In particular, given a  Brownian motion $W_t$ and denoting by $W'_t$  the stochastic process solution to

\begin{equation*}
	dW'_t=\sqrt{\eta\left(X_t\right)}dW_t,
\end{equation*}
it is possible to prove that $H_\eta\left(W'\right)$ is also a Brownian motion.

The total action of the random time change transformation on the coefficients of an SDE is described by the following proposition.

\begin{proposition}\label{Prop2}
Let $\eta : M \to \mathbb{R}_+$ be a smooth and strictly positive function and let $\left(X,W\right)$ be a solution to the SDE $\left(\mu,\sigma\right)$. Then the process $\left(H_\eta\left(X\right),H_\eta \left(W'\right)\right)$ is solution to the SDE $\left(\mu',\sigma'\right)$ with
	\begin{equation*}
	\begin{split}
	\mu'&=\frac{1}{\eta}\mu\\
	\sigma'&=\frac{1}{\sqrt{\eta}}\sigma.
	\end{split}
	\end{equation*}
\end{proposition}

\subsection*{Random rotations}

Since Brownian motion is invariant with respect to random rotations (see \cite{DMU} and \cite{AlDeMoUg2018} for the general concept of gauge transformations), and the solution to the SDE is understood in a weak sense, an additional possible   transformation of the SDE consists in a random rotation of the driving Brownian motion.

By using the L\'evy characterization of Brownian motion the following result states how an SDE is modified by a random rotation of its driving process.

\begin{proposition}
	Let $B : M \to SO \left(m\right)$ be a smooth function and let $\left(X,W\right)$ be a solution to the SDE $\left(\mu,\sigma\right)$. Then $\left(X,W'\right)$, where
	\begin{equation*}
	dW'_t=B\left(X_t\right)\cdot dW_t,
	\end{equation*}
	is a solution to the SDE $\left(\mu',\sigma'\right)$ with
	\begin{equation*}
	\begin{split}
	\mu'&=\mu,\\
	\sigma'&=\sigma \cdot B^{-1}.
	\end{split}
	\end{equation*}
	
\end{proposition}

\subsection*{Stochastic transformations}

Collecting the  natural transformations introduced above, we can  give the following definition.

\begin{definition}\label{stochastic_transformation}
Given a diffeomorphism $\Phi:M \rightarrow M'$ and two smooth functions  $B:M \rightarrow SO(m)$ and  $\eta:M \rightarrow \mathbb{R}_+$, the triad $T=(\Phi,B,\eta)$ is called a \emph{stochastic transformation}.
\end{definition}
We can distinguish two actions of the stochastic transformation $T=(\Phi,B,\eta)$: one on the solution processes  and the other one on the coefficients of an SDE.
\begin{definition}\label{definition1}
If $(X,W)$ is
a solution process,  we define the transformed process $P_T(X,W)=(X',W')$ as
\begin{eqnarray*}
X'=\Phi(H_{\eta}(X))\\
W'=H_{\eta}(\tilde{W}),
\end{eqnarray*}
where
$$\tilde{W}^{\alpha}_t=\int_0^t{ B^{\alpha}_{\beta}(X_s) dW^{\beta}_s}.$$
\end{definition}
\begin{definition}\label{definition2}
If $(X,W)$ is a solution process of an SDE $(\mu, \sigma)$, we define the transformed SDE $E_T(\mu,\sigma)=(\mu',\sigma')$ as
\begin{eqnarray*}
\mu'&=&\left(\frac{1}{\eta}L(\Phi)\right) \circ \Phi^{-1}\\
\sigma'&=&\left(\frac{1}{\sqrt{\eta}} \nabla\Phi \cdot \sigma \cdot B^T \right) \circ \Phi^{-1}.
\end{eqnarray*}
\end{definition}
By composing the previous propositions one can prove the following result.
\begin{proposition}\label{Prop3}
If $T$ is a stochastic transformation and $(X,W)$ is a solution to the SDE $(\mu,\sigma)$, then $P_T(X,W)$ is a solution to the SDE $E_T(\mu,\sigma)$.
\end{proposition}

Since the set of all stochastic transformations is a group with  composition law given by
$$T' \circ T=(\Phi' \circ \Phi,(B'\circ \Phi) \cdot B, (\eta' \circ \Phi ) \eta)$$
and unit $1_M=(id_M,I_m, 1)$, we can consider a one parameter group $T_a$ of stochastic transformations such that $T_a \circ T_b=T_{a+b}$ and $T_0=1_M$ and the corresponding infinitesimal generator $V$.
\begin{definition}
A triad $V=(Y,C,\tau)$, where  $ Y $ is a vector field on $ M $, $ C :M \to \mathfrak{so}\left(m\right) $ and $ \tau :M \to \mathbb{R} $ are smooth functions, is called an \emph{infinitesimal stochastic transformation} if
\begin{equation}\label{equation_infinitesimal_SDE2}
\begin{aligned}
Y(x)&=\partial_a(\Phi_a(x))|_{a=0}&\\
C(x)&=\partial_a(B_a(x))|_{a=0}&\\
\tau(x)&=\partial_a(\eta_a(x))|_{a=0}&
\end{aligned}
\end{equation}
\end{definition}
On the other side, considering the triad  $(Y, C,\tau)$ defined as above, the solution $\left(\Phi_a,B_a,\eta_a\right)$ to the equations
\begin{equation}
\begin{aligned}
\partial_a\left(\Phi_a\left(x\right)\right)&=Y\left(\Phi_a\left(x\right)\right)\\
\partial_a\left(B_a\left(x\right)\right)&=C\left(\Phi_a\left(x\right)\right)\cdot B_a \left(x\right)\\
\partial_a\left(\eta_a\left(x\right)\right)&=\tau\left(\eta_a\left(x\right)\right) \eta_a\left(x\right)
\end{aligned}
\end{equation}
with initial data $\Phi_0\left(x\right)=id_M\left(x\right)$, $B_0\left(x\right)=I_m$, $\eta_0\left(x\right)=1$, is a one-parameter group of stochastic transformations.

\subsection*{Weak symmetries and determining equations}

The introduction of a (global) stochastic transformation of the solution process to an SDE allows us  to  provide a good definition of (weak) symmetry for the SDE.

\begin{definition}
A stochastic transformation  $T$ is a \emph{(finite weak) symmetry} of an SDE $(\mu,\sigma)$ if, for every solution process $(X,W)$, $P_T(X,W)$ is a solution process to the same SDE.
\end{definition}

The following characterization of the symmetry property in terms of the coefficients of the SDE holds.

\begin{proposition}\label{Prop4}
A  stochastic transformation $T=\left(\Phi,B,\eta\right)$ is a (finite weak) symmetry of an  SDE $\left(\mu,\sigma\right)$ if and only if	
\begin{equation}\label{carCoef}
\begin{split}
\left(\frac{1}{\eta}L\left(\Phi\right)\right)\circ \Phi^{-1}=&\mu \\
\left(\frac{1}{\sqrt{\eta}}\nabla\Phi \cdot \sigma \cdot B^{-1}\right)\circ\Phi^{-1}&=\sigma.
\end{split}
\end{equation}
\end{proposition}

An infinitesimal stochastic transformation $V$  generating a one parameter group $T_a$ is called an \emph{infinitesimal symmetry} of the SDE $(\mu, \sigma)$ if $T_a$ is a (finite) symmetry of $(\mu, \sigma)$.

A useful characterization of the infinitesimal symmetries of an SDE is provided by the following statement.

\begin{theorem}\label{determiningeqs}
An infinitesimal stochastic transformation $V$ is an (infinitesimal weak) symmetry of the SDE $(\mu,\sigma)$ if and only if  $V$ generates a one-parameter group on $M$ and the following \emph{determining equations} hold
\begin{eqnarray}\label{determining_eq}
[Y,\sigma]&=&-\frac{1}{2} \tau \sigma-\sigma \cdot C\\
Y(\mu)-L(Y)&=&-\tau \mu,
\end{eqnarray}
where $[Y, \sigma]^i_{\alpha}= Y(\sigma^i_{\alpha})- \partial_k(\Phi^i)\sigma_{\alpha}^k$.

\end{theorem}

The determining equations obtained in Theorem \ref{determiningeqs} provide a powerful tool for the explicit computation of  the symmetries of an SDE. Indeed, they are definitely easier to solve  than the equations introduced in Proposition \ref{Prop4} for finite symmetries.

\section{An additional random transformation: a probability measure change \label{section_2}}

\subsection{Extended stochastic transformations}

In this section we  introduce an additional transformation of an SDE and of its solution process given by the change of measure through Girsanov theorem. We propose here only a short introduction of the probabilistic theory of absolutely continuous predictable change of measure of Brownian motion. The interested reader is referred, e.g., to \cite{RoWi2000} Section 38 and to \cite{Oksendal} Section 8.6 for an extended treatment of the subject.\\

Let $\left(\Omega,\mathcal{F}, (\mathcal{F}_t), \mathbb{P}\right)$ be a filtered probability space where an $ (\mathcal{F}_t)$-Brownian motion $W_t$ is defined. We fix a finite time horizon $I=[0,\mathcal{T}]$ with $\mathcal{T}>0$ and we work  with the filtration $\mathcal{F}_{\mathcal{T}}=\{\mathcal{F}_t, 0 < t< \mathcal{T}\}$.\\
Given a process $\left(\theta_s\right)_{s\in[0,\mathcal{T}]} \in M^2_{loc}[0,\mathcal{T}] $  , we define the process  $\left(Z_t\right)_{t\in[0,\mathcal{T}]}$ by setting
\begin{equation}\label{supermartingale}
Z_t=Z_t(\theta):=\exp\left\{\int_{0}^{t}\theta_sdW_s -\frac{1}{2}\int_{0}^{t}\theta_s^2ds \right\}.
\end{equation}
An application of It\^o formula gives
\begin{equation*}
dZ_t=Z_t\theta_tdW_t,
\end{equation*}
which says that $Z$ is a local martingale.  Since $Z_t$ is strictly positive, we have that $Z_t$ is always a supermartingale. The process $Z_t$ is called \emph{exponential supermartingale} and can be used as a Radon-Nikodym derivative of a new probability measure $\mathbb{Q}$. This means that we can define on $\Omega$ the probability law $\mathbb{Q}_\mathcal{T}(d\omega):=Z_\mathcal{T}(\omega)\mathbb{P}_\mathcal{T}(d\omega)$, i.e.
 $\mathbb{Q}_\mathcal{T}(A):=\int_A{Z_\mathcal{T}(\omega)d\mathbb{P}_\mathcal{T}(\omega)}, \forall A\in \mathcal{F}$.
We recall the fundamental result allowing the change of a probability measure into an equivalent one.
\begin{theorem}[Girsanov's theorem]
Let $\left(Z_t\right)_{t\in[0,\mathcal{T}]}$ be the exponential supermartingale defined  in \eqref{supermartingale}. If $\left(Z_t\right)_{t\in[0,\mathcal{T}]}$ is a $\mathbb{P}$-martingale, then the process  $\left(W'_t\right)_{t \in [0,\mathcal{T}]}$ given by
\begin{equation*}
	W'_t=W_t-\int_{0}^{t} \theta_sds
\end{equation*}
is an $ (\mathcal{F})_t$-Brownian motion with respect to the probability measure  $\mathbb{Q}$, where
\[ \frac{\md\mathbb{Q}_\mathcal{T}}{\md\mathbb{P}_\mathcal{T}}:=\left.\frac{\md\mathbb{Q}}{\md\mathbb{P}}\right|_{\mathcal F_{\mathcal T}}=Z_\mathcal{T}.\]
\end{theorem}
\begin{proof}
The proof can be found for example in \cite{RoWi2000}, Theorem 38.5.
\end{proof}

In order to better explain the probabilistic meaning of the previous theorem, we recall the following notion.
\begin{definition}\label{equiv}
Let $\mathbb{P}$ and $\mathbb{Q}$ two probability measures on a measurable space  $\left(\Omega,\mathcal{F}\right)$.  $\mathbb{Q}$ is called \emph{absolutely continuous} with respect to $\mathbb{P}$, and we write $\mathbb{Q} \ll \mathbb{P}$, if
\begin{equation*}
	\mathbb{P}\left(A\right)=0  \Rightarrow \mathbb{Q}\left(A\right)=0 \text{ for all } A\in \mathcal{F}.
\end{equation*}
The measures $\mathbb{P}$ and $\mathbb{Q}$ are \emph{equivalent} if $\mathbb{Q} \ll \mathbb{P}$ and $\mathbb{P} \ll \mathbb{Q}$, that is if  they have the same set of null events.
\end{definition}
Since   the Radon-Nikodym derivative   in Girsanov theorem is strictly positive, i.e. $\frac{d\mathbb{Q}_\mathcal{T}}{d\mathbb{P}_\mathcal{T}}=Z_\mathcal{T}>0$, one can prove that the measure  $\mathbb{Q}_\mathcal{T}$ is actually equivalent to  $\mathbb{P}_\mathcal{T}$.\\

In order to be able to apply Girsanov theorem we need  $Z_t$ to be a  martingale. Establishing when the supermartingale $Z_t$ is a (global) martingale is in general not an easy task. Instead of using the well known Novikov condition (see \cite{ReYo1999}, Chapter VIII, Proposition 1.14), we introduce the following definition and lemma  which  allow us to achieve the same goal.

\begin{definition}
A smooth SDE $(\mu,\sigma)$ is called \emph{non explosive} if any solution $(X,W)$ to $(\mu,\sigma)$ is defined for all times $t\geq 0$.
A smooth vector field $h$ is called \emph{non explosive} for the non explosive SDE $(\mu,\sigma)$ if the SDE $(\mu+\sigma \cdot h,\sigma)$ is a non explosive SDE.
A positive smooth function $\eta$ is called \emph{a non explosive time change} for the non explosive SDE $(\mu,\sigma)$ if the SDE $\left(\frac{\mu}{\eta},\frac{\sigma}{\sqrt{\eta}} \right)$ is non explosive.
\end{definition}

\begin{lemma}\label{lemma_change}
Suppose that the equation $(\mu,\sigma)$ is a non explosive SDE, let $h:M\rightarrow \mathbb{R}^n$ be a smooth non explosive vector field, and suppose that $(X,W)$ is a solution to the SDE $(\mu,\sigma)$. Then the exponential supermartingale $Z_t$ associated with $\theta_t=h(X_t)$ is a (global) martingale.
\end{lemma}
\begin{proof}
The proof can be found in \cite{ElJaLi2010}, Theorem 9.1.4.
\end{proof}

\begin{remark}
It is simple to see that, if $(\mu,\sigma)$ is a non explosive SDE and $T=(\Phi,B,1)$ is a stochastic transformation, then also $P_T(\mu,\sigma)$ is non explosive.
\end{remark}

\begin{theorem}\label{Girsanovderivative}
Let $(X,W)$ be a solution to the non explosive SDE $(\mu,\sigma)$ on the probability space $(\Omega,\mathcal{F},\mathbb{P})$ and let $h$ be a smooth non explosive vector field for $(\mu,\sigma)$.  Then $(X,W')$ is a solution to the SDE $(\mu',\sigma')=(\mu+\sigma \cdot h,\sigma)$ on the probability space $(\Omega,\mathcal{F},\mathbb{Q})$ where
\begin{eqnarray*}
W'_t&=&-\int_0^t{h(X_s)ds}+W_t\\
\frac{\md \mathbb{Q}_\mathcal{T}}{\md \mathbb{P}_\mathcal{T}}&=&\exp{\int_0^\mathcal{T}{h_{\alpha}(X_s)dW^{\alpha}_s}-\frac{1}{2}\int_0^\mathcal{T}{\sum_{\alpha=1}^m(h_{\alpha}(X_s))^2ds}}.
\end{eqnarray*}
\end{theorem}
\begin{proof}
The proof is a simple application of Girsanov theorem and Lemma \ref{lemma_change}.
\end{proof}

We define a new (finite) stochastic transformation including, besides the three transformations described in the previous section, also a change of the underlying probability measure.

\begin{definition}
	Let $\Phi : M \to M'$ be a diffeomorphism and let $B : M \to ~ SO\left(m\right)$, $\eta : M \to \mathbb{R}_+$ and $h : M \to \mathbb{R}^m$ be smooth functions. We call $T:=\left(\Phi,B,\eta,h\right)$ a \emph{ (weak finite) extended stochastic transformation}. 
\end{definition}

Since we cannot apply the previous transformation to any SDE, we give the following definition.

\begin{definition}
Let  $T=\left(\Phi,B,\eta,h\right)$ be an extended stochastic transformation. If the pair $ \left(X,W\right) $ is a continuous stochastic process where $X$ takes values in  $M$ and $W$ is an $m$-dimensional Brownian motion in the space $\left(\Omega,\mathcal{F},\mathbb{P}\right)$ and  $ \left(X,W\right) $ is a solution to the non explosive SDE $(\mu,\sigma)$ for which $h$ and $\eta$ are non explosive, we can define the process $P_T\left(X,W\right)=\left(P_T\left(X\right),P_T\left(W\right)\right)$, where $P_T\left(X\right)$ takes values in $M'$ and $P_T\left(W\right)$ is a Brownian motion into the space $\left(\Omega,\mathcal{F},\mathbb{Q}\right)$. The process components are given by
\begin{equation*}
\begin{split}
P_T\left(X\right)&= \Phi\left(H_\eta\left(X\right)\right)\\
dW'_t&=\sqrt{\eta\left(X_t\right)}B\left(X_t\right)\left(dW_t-h\left(X_t\right)dt\right)\\
	P_T\left(W\right)&= H_\eta\left(W'\right)\\
\frac{\md \mathbb{Q}_{\mathcal{T}}}{\md \mathbb{P}_{\mathcal{T}}}&=\int_0^\mathcal{T}{h_{\alpha}(X_s)dW^{\alpha}_s}-\frac{1}{2}\int_0^\mathcal{T}{\sum_{\alpha=1}^m(h_{\alpha}(X_s))^2ds}.
\end{split}
\end{equation*}
We call $P_T\left(X,W\right)$ the \emph{transformed process} of $ \left(X,W\right) $ with respect to $ T $ and we call the function $ P_T $ the \emph{process transformation} associated with $ T $.
\end{definition}

The previous definition is necessary if we are interested in defining the transformation $P_T$ on the set of solutions $(X,W)$ to the SDE $(\mu,\sigma)$. If we focus only on the SDE identified with the pair of smooth functions $(\mu, \sigma)$ and on the action $E_T$ of the stochastic transformation $T$, the previous definition is no more necessary and we can define the transformed SDE $E_T(\mu,\sigma)$ without making any request on the non-explosiveness of the process $(X,W)$.

\begin{definition}\label{sdetransformed}
Let $T=\left(\Phi,B,\eta,h\right)$ be an extended stochastic transformation. If $ \left(\mu,\sigma\right) $ is an  SDE on $ M $ then we define  $E_T\left(\mu,\sigma\right):=\left(E_T\left(\mu\right),E_T\left(\sigma\right)\right)$ the SDE on $M'$ defined as
	\begin{equation*}
	\begin{split}
	E_T\left(\mu\right)&=\left(\frac{1}{\eta}\left[L\left(\Phi\right)+\nabla \Phi \cdot \sigma \cdot h\right]\right)\circ \Phi^{-1}\\
	E_T\left(\sigma\right)&=\left(\frac{1}{\sqrt{\eta}}\nabla \Phi \cdot \sigma \cdot B^{-1}\right) \circ \Phi^{-1}.
	\end{split}
	\end{equation*}
	We call  $ E_T\left(\mu,\sigma\right) $ the \emph{transformed SDE} of $ \left(\mu,\sigma\right) $ with respect to $ T$ and we call the map $E_T$ the \emph{SDE transformation} associated with $T$.
\end{definition}

We stress that a specific choice on the order according to which the transformations are applied was performed.  First the change of measure and then the rotation together with the random time change. While the diffusion coefficient does not change, the form of the transformed drift strongly depends on this choice.

\begin{theorem}\label{theorempt}
Let $T=\left(\Phi,B,\eta,h\right)$ be an extended stochastic transformation and let $ \left(X,W\right) $ be a solution to the non explosive SDE $\left(\mu,\sigma\right)$ such that $E_T(\mu,\sigma)$ is non explosive. Then $P_T\left(X,W\right)$ is solution to the SDE $E_T\left(\mu,\sigma\right).$
\end{theorem}
\begin{proof}
	We show how the SDE $(\mu, \sigma)$ changes when we apply the transformations in the order specified above. For simplicity we omit the explicit dependence on the values of the process. The starting equation is
	\begin{equation*}
	dX_t=\mu dt + \sigma \cdot dW_t
	\end{equation*}
	We apply first the measure change obtaining
	\begin{equation*}
		dX'_t=\left(\mu+\sigma \cdot h\right) dt + \sigma \cdot dW'_t
	\end{equation*}
	and then the rotation of the Brownian motion and the random time change, obtaining
	\begin{equation*}
	\begin{split}
	dX''_t&=\frac{1}{\eta}\left(\mu+\sigma \cdot h\right) dt + \frac{1}{\sqrt{\eta}}\sigma \cdot B^{-1} \cdot dW''_t\\
	&=\frac{1}{\eta}\tilde{\mu}dt + \tilde{\sigma} \cdot dW''_t \\
		\end{split}
	\end{equation*}
	where we set $\tilde{\mu}=\mu+\sigma \cdot h$ and $\tilde{\sigma}=\frac {1}{\sqrt{\eta}}\sigma \cdot B^{-1}$. Finally, we apply the spatial transformation considering only the drift (since the equations involving $\sigma$ coincide with the analogous equations obtained considering non-extended stochastic transformation), and we get
	\begin{equation*}
	\begin{split}
	E_T\left(\mu\right)&=\left(\frac{1}{\eta}\left[\frac{1}{2}\left(\sigma \cdot \sigma^T\right)^{ij} \partial_i \partial_j \Phi +  \tilde{\mu}^i\partial_i \Phi\right]\right)\circ \Phi^{-1}\\
	&=\left(\frac{1}{\eta} \left[L\left(\Phi\right)+\nabla \Phi \cdot \sigma\cdot h\right] \right)\circ \Phi^{-1}.
	\end{split}
	\end{equation*}
	Exploiting Definition \ref{sdetransformed} and Theorem \ref{theorempt}, this concludes the proof.
	\end{proof}

In the following, in order to simplify notations, we refer to extended stochastic transformations just as stochastic transformations.

\subsection{Stochastic transformations group and associated infinitesimal transformations}\label{gH}

Let $G=SO(m)\times \mathbb{R}_+\ltimes\mathbb{R}^m $ be the group of rototranslations with a scaling factor whose elements $g=(B,\eta, h) $ can be identified with the matrices of the form
 \[ \left(  \begin{array}{cc}
    \frac {B^{-1}}{\sqrt \eta} & h \\
    0 & 1 \\
  \end{array}
\right)\]
 If we consider the trivial principal bundle $\pi:M \times G \to M$ with structure group $G$, we can define the following action of $G$ on $M\times G$
\begin{eqnarray*}
	R_{g_2}: &M\times G &\to M\times G\\
	&(x,g_1) &\to (x, g_1 \cdot g_2)
	\end{eqnarray*}
which leaves $M$ invariant, where $g_1 \cdot g_2$ denotes  the standard product $g_1\cdot g_2=(B_1,\eta_1,h_1)\cdot (B_2,\eta_2,h_2)=(B_2 B_1, \eta_2\eta_1, \frac {B_1^{-1}}{\sqrt{\eta_1}} h_2+h_1)$.
If we consider another trivial principal bundle $\pi':  M'\times G \to M'$,  we say that a diffeomorphism $F: M\times G  \to M'\times G $ between the two  principal  bundles $M\times G $ and $ M'\times G $ is an \emph{isomorphism} if $F$ preserves the structures of principal bundles of both $M\times G $ and $ M'\times G $, i.e. there exists a diffeomorphism $\Phi: M \to M'$ such that
\begin{eqnarray*}
	F\circ \pi' &=&\pi \circ \Phi\\
	F\circ R_g &=& R_g\circ F
\end{eqnarray*}
for any $g\in G$.
It is easy to check that an isomorphism of the previous form is completely determined by its value on $(x, e)$, where $e$ is the unit element of $G$. Therefore, there is a natural identification between a stochastic transformation $T=(\Phi, B, \eta, h)$ and the isomorphism $F_T$ such that $F_T(x,e)=(\Phi(x), g)$ where $g=(B,\eta,h)$.
In particular we can exploit the natural composition of isomorphisms in order to define a composition of stochastic transformations in the following way: if $T_1=\left(\Phi_1,B_1,\eta_1,h_1\right)$ and $T_2=\left(\Phi_2,B_2,\eta_2,h_2\right)$ then
	\begin{equation*}
	T_2\circ T_1=\left(\Phi_2 \circ \Phi_1,\left(B_2 \circ \Phi_1\right)\cdot B_1,\left(\eta_2 \circ \Phi_1\right)\eta_1,\frac{1}{\sqrt{\eta_1}}B_1^{-1} \cdot \left(h_2 \circ \Phi_1\right)+h_1\right).
	\end{equation*}
Moreover we can consider the inverse transformation of  $T=\left(\Phi,B,\eta,h\right)$ given by
\begin{equation*}
T^{-1}=\left(\Phi^{-1},\left(B\circ \Phi^{-1}\right)^{-1},\left(\eta\circ \Phi^{-1}\right)^{-1}, -\sqrt{\eta}B \cdot h\circ \Phi^{-1}\right).
\end{equation*}
The following theorem, generalizing a result obtained in \cite{DMU} for stochastic transformations without the change of probability measure, shows that the  identification of  stochastic transformations with the isomorphisms of a suitable trivial principal bundle has a deep probabilistic counterpart in terms of SDE and process transformations.

\begin{theorem}\label{compMis}
	Let $T_1,T_2$ be two  stochastic transformations, let $(\mu,\sigma)$ be a non explosive SDE such that $E_{T_1}(\mu,\sigma)$ and $E_{T_2}(E_{T_1}(\mu,\sigma))$ are non explosive and let $(X,W)$ be a solution to the SDE $(\mu,\sigma)$ on the probability space  $\left(\Omega,\mathcal{F},\mathbb{P}\right)$.  Then on the probability space  $\left(\Omega,\mathcal{F},\mathbb{Q}\right)$ we have
	\begin{eqnarray*}
	P_{T_2}(P_{T_1}(X,W))&=&P_{T_2 \circ T_1}(X,W)\\
	E_{T_2}(E_{T_1}(\mu,\sigma))&=&E_{T_2 \circ T_1}(\mu,\sigma).
	\end{eqnarray*}
\end{theorem}
\begin{proof}
The proof of this result without the change of probability measure can be found in \cite{DMU}, hence  we have just to add the part related to the last component of the stochastic transformation.\\
Applying the two transformations to the Brownian motion, according to the fixed specified order, we obtain
	\begin{equation*}
	\begin{split}
		dW'_t&=\sqrt{\eta_1}B_1\cdot \left(dW_t-h_1dt\right)\\
			dW''_t&=\sqrt{\eta_2}B_2 \cdot \left(dW'_t-h_2dt\right)\\
			&=\sqrt{\eta_2}B_2 \cdot\left(\sqrt{\eta_1}B_1\left(dW_t-h_1dt\right)-h_2dt\right)\\
			&=\sqrt{\eta_2 \eta_1}B_2 \cdot B_1\cdot\left(dW_t-\left(h_1+\frac{1}{\sqrt{\eta_1}}B_1^{-1}\cdot h_2\right)dt\right)\\
			&=\sqrt{\tilde{\eta}} \tilde{B} \cdot \left(dW_t-\tilde{h}dt\right),
			\end{split}
		\end{equation*}
		where
		\begin{equation*}
		\tilde{h}=\frac{1}{\sqrt{\eta_1}}B_1^{-1}\cdot h_2+h_1.
		\end{equation*}
		Noting that, after the first transformation $T_1$, the state variable is $\Phi \left(X\right)$, we obtain the statement.		
\end{proof}

Since the set of stochastic transformations is a group with respect the composition $\circ$, we can  consider the one parameter group $T_a=\left(\Phi_a,B_a,\eta_a,h_a\right)$ and the corresponding infinitesimal transofrmation  $V=\left(Y,C,\tau,H\right)$ obtained in the usual way
\begin{equation*}
\begin{split}
Y\left(x\right)&=\partial_a\left(\Phi_a\left(x\right)\right)|_{a=0}\\
C\left(x\right)&=\partial_a\left(B_a\left(x\right)\right)|_{a=0}\\
\tau\left(x\right)&=\partial_a\left(\eta_a\left(x\right)\right)|_{a=0}\\
H\left(x\right)&=\partial_a\left(h_a\left(x\right)\right)|_{a=0}.\\
\end{split}
\end{equation*}
On the other hand, given $V=\left(Y,C,\tau,H\right)$, where $Y$ is a vector field on $M$, $C : M\to \mathfrak{so}\left(m\right)$,  $\tau : M \to \R$  and $H : M \to \R^m$ are smooth functions, we can  reconstruct the one parameter transformation group $T_a$ exploiting the following relations

\begin{equation*}
\begin{split}
\partial_a\left(\Phi_a\left(x\right)\right)&=Y\left(\Phi_a\left(x\right)\right)\\
\partial_a\left(B_a\left(x\right)\right)&=C\left(\Phi_a\left(x\right)\right)\cdot B_a\left(x\right)\\
\partial_a\left(\eta_a\left(x\right)\right)&=\tau\left(\Phi_a\left(x\right)\right)\eta_a\left(x\right)\\
\end{split}
\end{equation*}	
with initial data $\Phi_0\left(x\right)=id_M\left(x\right)$, $B_0\left(x\right)=I_m$, $\eta_0\left(x\right)=1$. Moreover, using Theorem \ref{compMis} and the properties of the flow,  we obtain
\begin{equation*}
\begin{split}
h_{b+a}\left(x\right)&=\frac{1}{\sqrt{\eta_a \left(x\right)}}B_a^{-1}\left(x\right)\cdot h_b \left(\Phi_a\left(x\right)\right)+h_a\left(x\right)\\
\partial_b \left(h_{b+a}\left(x\right)\right)&=\frac{1}{\sqrt{\eta_a \left(x\right)}}B_a^{-1}\left(x\right) \cdot\partial_b \left( h_b \left(\Phi_a\left(x\right)\right)\right)\\
\partial_b \left(h_{b+a}\left(x\right)\right)|_{b=0}&=\partial_a \left(h_a\left(x\right)\right)=\frac{1}{\sqrt{\eta_a\left(x\right)}}B_a^{-1}\left(x\right)\cdot H \left(\Phi_a\left(x\right)\right)
\end{split}
\end{equation*}
subjected to the initial condition $h_0\left(x\right)=0$. \\

The notions of composition and one parameter group of  stochastic transformations allow us to consider  the action of a finite stochastic transformation on an infinitesimal one, as well as the Lie brackets between two infinitesimal stochastic transformations. Let $T=(\Phi,B,\eta,h)$ be a finite stochastic transformation and let $V=(Y,C,\tau,H)$ be an infinitesimal  stochastic transformation with associated one parameter group given by $T_a$. Then there exists an infinitesimal  stochastic transformation $T_*(V)$, called \emph{push forward of $V$ through the transformation $T$}, generating the one parameter group given by $T\circ T_a \circ T^{-1}$ and having the following form
\begin{multline}
T_*(V)=\left(\Phi_*(Y),(B\cdot C \cdot B^{-1}+Y(B) \cdot B^{-1})\circ \Phi^{-1},\left(\tau+\frac{Y(\eta)}{\eta}\right)\circ \Phi^{-1},\phantom{\frac{\tau}{2}}\right.\\
\left.\phantom{\frac{Y(\eta)}{\eta}}\left(-\sqrt{\eta}B\cdot\left(\frac{\tau}{2}+C \right)\cdot h+\sqrt{\eta} B\cdot H + \sqrt{\eta} B\cdot Y(h) \right)\circ \Phi^{-1}      \right),
\end{multline}
where $\Phi_*(Y)$ denotes the push forward of the vector field $Y$. Using the previous expression we can define the Lie brackets between two infinitesimal stochastic transformations $V_1=(Y_1,C_1,\tau_1,H_1)$ and $V_2=(Y_2,C_2,\tau_2,H_2)$ which is
\begin{multline}
[V_1,V_2]=\left([Y_1,Y_2], Y_1(C_2)-Y_2(C_1) -\{C_1,C_2\}, Y_1(\tau_2)-Y_2(\tau_1)  \right.\\
\left. Y_1(H_2)-Y_2(H_1)+\left(\frac{\tau_1}{2}+C_1 \right)\cdot H_2 - \left(\frac{\tau_2}{2}+C_2 \right)\cdot H_1\right).
\end{multline}

\subsection{Determining equations}\label{V}

In this section we provide the new \emph{determining equations} (including the probability measure change) for the infinitesimal symmetries of an SDE. \par
\begin{proposition}
An (extended) stochastic transformation $T=(\Phi,B,\eta,h)$ is a (finite) symmetry of the non explosive SDE $(\mu,\sigma)$ if and only if
\begin{eqnarray*}
\left(\frac{1}{\eta} \left[L\left(\Phi\right)+\nabla \Phi \cdot \sigma\cdot h\right] \right)\circ \Phi^{-1}&=&\mu\\
\left(\frac{1}{\sqrt{\eta}}\nabla\Phi \cdot \sigma \cdot B^{-1}\right)\circ\Phi^{-1}&=&\sigma.
\end{eqnarray*}
\end{proposition}
\begin{proof}
The proof is a simple generalization of the proof of Proposition \ref{Prop4} (see \cite{DMU}) using Theorem \ref{theorempt} in replacement of Proposition \ref{Prop3}. It is important to note that since $E_T(\mu,\sigma)=(\mu,\sigma)$, in order to use Lemma  \ref{lemma_change}, we need only to require that $(\mu,\sigma)$ is non explosive.
\end{proof}

The characterization expressed by Proposition \ref{Prop4} for the coefficients $E_T\left(\mu\right)$ and $E_T\left(\sigma\right)$ under the symmetry hypothesis still holds because the new probability measure $\mathbb{Q}$ is equivalent to $\mathbb{P}$ (see Definition \ref{equiv}).

\begin{theorem}\label{theorem_determining}
	Let $V=\left(Y,C,\tau,H\right)$ be  an infinitesimal stochastic transformation. Then $V$ is an infinitesimal symmetry of the SDE $\left(\mu,\sigma\right)$ if and only if $V$ generates a one parameter group defined on $M$ and we have
	\begin{equation*}
	\begin{split}
	Y\left(\mu\right)-L\left(Y\right)-\sigma \cdot H+\tau \mu&=0\\
	\left[Y,\sigma \right] +\frac{1}{2}\tau \sigma + \sigma \cdot C&=0.
	\end{split}
	\end{equation*}
\end{theorem}
	\begin{proof}
	We prove only that the property of  $V$ to be an infinitesimal symmetry is a sufficient condition for the validity of determining equations.
	The one parameter group $T_a=\left(\Phi_a,B_a,\eta_a,h_a\right)$ associated with the infinitesimal transformation  $V=\left(Y,C,\tau,H\right)$ is a symmetry if and only if  the following equations are satisfied
	\begin{equation*}
	\begin{split}
		\left(\frac{1}{\eta_a}\left[L\left(\Phi_a\right)+\nabla\Phi_a \cdot \sigma \cdot h_a\right]\right)\circ \Phi^{-1}_a&=\mu,\\
		\left(\frac{1}{\sqrt{\eta_a}}\nabla \Phi_a \cdot \sigma \cdot B^{-1}_a\right) \circ \Phi^{-1}_a&=\sigma.
		\end{split}
	\end{equation*}
	The first equation is equivalent to
	\begin{equation*}
		L\left(\Phi_a\right)+\nabla \Phi_a \cdot \sigma \cdot h_a =\eta_a\mu\left(\Phi
		_a\right).
	\end{equation*}
	Deriving both sides with respect to the parameter $a$ and then evaluating them in $a=0$, we obtain
	\begin{equation*}
	\begin{split}
	&L\left(\partial_a\left(\Phi_a\right)|_{a=0}\right)+\partial_a\left(\nabla \Phi_a \cdot \sigma \right)|_{a=0}\cdot h_0+\nabla \Phi_0 \cdot \sigma  \cdot \partial_a\left(h_a\right)|_{a=0}=\\
	&\partial_a\left(\eta_a\right)|_{a=0}\; \mu \left(\Phi_0\right)+\eta_0\;  \partial_a\left(\mu\left(\Phi_a\right)\right)|_{a=0}
	\end{split}
	\end{equation*}
	and therefore the first determining equation
	\begin{equation*}
L\left(Y\right)+\sigma \cdot H=\tau \mu+	Y\left(\mu\right).
	\end{equation*}
The second equation, not depending on the measure change, coincides with the analogous equation obtained in Theorem \ref{determiningeqs} and this concludes the proof.
\end{proof}

\section{ Doob transformations and symmetries of Kolmogorov equation \label{section_3}}

\subsection{Lie point symmetries of Kolmogorov equation}

In the previous section we showed how the determining equations modify if we include in the stochastic transformation the change of the probability measure underlying the SDE. A natural question  arising in this setting  is the comparison between the determining equations  for the infinitesimal symmetries of an SDE introduced in Theorem \ref{theorem_determining} and the determining equations for the symmetries of the associated  Kolmogorov equation. \\
The symmetries of the Kolmogorov equations usually involves transformations explicitly depending on time that would be lost if we consider only autonomous transformations for our system. For this reason we have to extend our analysis in order to permit also non-autonomous transformations.
In order to achieve this, we use the standard trick, well known in the literature of dynamical system, for transforming a non-autonomous system into an autonomous one.
Indeed, given an SDE of the form \eqref{eqSDE}, we introduce a trivial  extra component $dZ_t=dt$, i.e., we consider the system

\begin{equation}\label{eqTIMESDE}
\begin{split}
dX^i_t  & = \mu^i(X_t) dt+\sigma^i_{\alpha} (X_t) dW^{\alpha}_t  \\
dZ_ t &  = dt
\end{split}
\end{equation}
In this setting, the Kolmogorov equation associated with \eqref{eqSDE} can be written as $L(u)=0$, where $L$ is the infinitesimal generator associated with \eqref{eqTIMESDE}, i.e.
\begin{equation}\label{KOLMOGOROV}
L(u)= A^{ij}\frac {\partial^2 u}{\partial x^i\partial x^j}+\mu^i \frac {\partial u}{\partial x^i} +\frac {\partial u}{\partial z}=0,
\end{equation}
and $A=\frac 12 \sigma\cdot \sigma^T$. In the following we always suppose that the matrix $A$ has constant rank.\\
Equation \eqref{KOLMOGOROV} is a second order partial differential equation in the independent variables $(x^i,z)$ and in the dependent variable $u$ describing the behavior of the mean value of regular functions of the solution process $X_t$. More precisely, a solution $u(x,z)$ to equation \eqref{KOLMOGOROV} is of the form $\mathbb{E}[f(X_\mathcal{T})|X_z=x]=u(x,z)$, with $z\in [0,\mathcal{T}]$ and $u(x, \mathcal{T})=f(x)$ (see, e.g., \cite{Oksendal} Section 8.1).\\

Since \eqref{KOLMOGOROV} is a PDE, it can be seen as a submanifold of the second order jet bundle $J^2(\mathbb{R}^{n+1},\mathbb{R})$ described by the equation $L(u)=0$ (see \cite{Olver} Chapter 2).
In this setting a Lie point infinitesimal symmetry for \eqref{KOLMOGOROV} is a vector field $\Xi$ on $J^0(\mathbb{R}^{n+1},\mathbb{R})$ such that
\begin{equation}\label{SYMMCOND}
  \Xi^{(2)}(L(u))|_{L(u)=0}=0,
\end{equation}
where $\Xi^{(2)}$ denotes the second order prolongation of the vector field $\Xi$.
If we rewrite \eqref{SYMMCOND} for a vector field of the form
\begin{equation}\label{VectorFieldXi}
\Xi=m(x,z)\frac{\partial}{\partial z} +\phi^i(x,z)\frac{\partial}{\partial x^i} +\Psi(x,z,u) \frac{\partial}{\partial u}
\end{equation}
we get the following determining equations for the Lie point infinitesimal symmetries of  \eqref{KOLMOGOROV}
\begin{eqnarray}
 \Psi(x,z,u) & = & -k(x,z)u+k_0(x,z)  \\ \label{eq:detPDE1}
 L(k_0) &= &0 \\ \label{eq:detPDE2}
 L(k)& = & 0 \\ \label{eq:detPDE3}
L(\phi)-\Xi(\mu)+2A \cdot \nabla k-L(m)\mu & = & 0\\ \label{eq:detPDE4}
L(m)A+\Xi(A)-\nabla \phi \cdot  A - A \cdot (\nabla \phi)^T & = & 0\\ \label{eq:detPDE5}
A \cdot \nabla m& = & 0 \label{eq:detPDE6}
\end{eqnarray}

We remark that the function $k_0$ only appears in the first two conditions and corresponds to the trivial symmetry $k_0(x,t)\frac {\partial }{\partial u}$ which takes into account the superposition principle for the solutions to any linear PDEs. Therefore, later on, we omit $k_0$  (and the corresponding condition $L(k_0)=0$) in the determining equations.\\

The comparison between the determining equations for the Lie point infinitesimal  symmetries of the PDE \eqref{KOLMOGOROV} and the determining equation for the stochastic symmetries for the SDE \eqref{eqTIMESDE} arising in Theorem \ref{theorem_determining} is not straightforward, first of all due to the different nature of the involved objects. Indeed, when we look for the infinitesimal symmetries of the PDE we look for a vector field of the form \eqref{VectorFieldXi}, while when we search symmetries for SDE \eqref{eqTIMESDE}, we deal with an infinitesimal transformation $V=(Y,C,\tau,H)$, where $Y$ is a vector field on $\mathbb{R}^{n+1}$, $C:\mathbb{R}^{n+1} \to \mathfrak{so}(m)$,  $\tau:\mathbb{R}^{n+1} \to \mathbb{R}$ and $H:\mathbb{R}^{n+1} \to \mathbb{R}^m$ are smooth functions.

In the following, in order to compare the determining equations arising in these two different settings we restrict the class of  transformations for the probability measure, introducing the notion of Doob transformations.

\subsection{Doob transformations}

The Doob transformation is a special kind of change of measure, and so a particular case of Girsanov transformation, in general related with  Markov processes. Originally introduced by Doob (see \cite{dooblibro}) it has been generalized in many directions  (see \cite{RoWi2000}  for the continuous case and \cite{doob} and  references therein for an overview on the topic).\\
In this  paper we only consider the original Doob transformation, which can be seen as a transformation of the infinitesimal generator $ L$ of a given diffusion process. In the conservative case (i.e. when the probability is conserved at each time) this transformation assumes the following simple form:
$$ L^{\varphi}=\varphi^{-1}L\varphi$$
with $\varphi$ a strictly positive function on the state space $M$. The new diffusion process corresponding to the infinitesimal generator $L^{\varphi}$ has a path measure (i.e. a measure on the pathspace of the process trajectories) which is absolutely continuous with respect to the original path measure with density given by the following  Radon-Nikodym derivative on the time interval $[0, \mathcal{T}]$:
\[\left.\frac{\md \mathbb{Q}_{L^{\varphi}}}{\md \mathbb{P}_{L}}\right|_{\mathcal{F}_\mathcal{T}}=\varphi^{-1}(X_0)\varphi(X_\mathcal{T}).\]

In the following we fix  a non explosive SDE $(\mu,\sigma)$ and a solution $(X,W)$ to it and we give a simplified definition of Doob transformation adapted to our setting.

\begin{definition}\label{Doobderivative}
Given a smooth function  $h:M\rightarrow \mathbb{R}^m$ which is non explosive with respect to $(\mu,\sigma)$, we say that $h$ is a \emph{Doob transformation characterized by the smooth function $\mathfrak{h}:M\rightarrow\mathbb{R}$} if the measure $\mathbb{Q}$ generated by the Girsanov transformation $h$ is such that
\[\frac{\md \mathbb{Q}_\mathcal{T}}{\md \mathbb{P}_\mathcal{T}}=\exp{\left(\mathfrak{h}(X_\mathcal{T})-\mathfrak{h}(X_0) \right)}.\]
\end{definition}
The next result provides suitable conditions on the functions $h$ and $\mathfrak{h}$ in order to guarantee that a Girsanov  transformation is also a Doob one.

\begin{proposition}
Let  $h:M\rightarrow \mathbb{R}^m$ be a smooth function associated with a Girsanov transformation of a non explosive SDE $(\mu, \sigma)$. If  $\mathfrak{h}:M\rightarrow \mathbb{R}^m$ is a smooth function satisfying
\begin{eqnarray}
h_{\alpha}(x)&=&\sigma_{\alpha}^i(x)\partial_{x^i}(\mathfrak{h})(x)\label{eq:Doob1}\\
\frac{1}{2}\sum_{\alpha=1}^m (h_{\alpha}(x))^2&=&-L(\mathfrak{h})(x)\label{eq:Doob2},
\end{eqnarray}
then $h$ is also a Doob transformation.
\end{proposition}
\begin{proof}
By equating  the Radon-Nikodym derivative given in Theorem \ref{Girsanovderivative} with the one given in Definition \ref{Doobderivative} we obtain the equality
\begin{equation}\label{hrelation}
\int_0^\mathcal{T}{h_{\alpha}(X_s)dW^{\alpha}_s}-\frac{1}{2}\int_0^\mathcal{T}{\sum_{\alpha=1}^m(h_{\alpha}(X_s))^2ds}=\mathfrak{h}(X_\mathcal{T})-\mathfrak{h}(X_0).
\end{equation}
Applying It\^o formula to the function $\mathfrak{h}$ we get
\begin{equation}
\mathfrak{h}(X_\mathcal{T})-\mathfrak{h}(X_0) =\int_0^\mathcal{T} L\mathfrak{h}(X_t)dt+\int_0^\mathcal{T} \nabla\mathfrak{h} (X_t)\sigma(X_t) dW_t.
\end{equation}
The proof follows by the uniqueness of canonical semimartingale  decomposition of continuous processes (see, e.g. Section 31 and Definition 31.3 in \cite{RoWi2000}) and by uniqueness of martingale representation theorem for processes adapted to Brownian filtrations  (see, e.g., Theorem 36.1 in \cite{RoWi2000}).
\end{proof}
\begin{remark}
It is important to note that the fact that a change of measure $h$ is a Doob transformation strongly depend on the SDE $(\mu,\sigma)$. Indeed, the right hand side of equation \eqref{eq:Doob1} depends explicitly on $\sigma$ and the right hand side of equation \eqref{eq:Doob2}  depends, through the operator $L$,  on both $\mu$ and $\sigma$.
\end{remark}

Relations \eqref{eq:Doob1} and \eqref{eq:Doob2} allow us to derive also the conditions under which a stochastic transformation $T=(\Phi,B,\eta,h)$ involves the Doob transformation associated with $\mathfrak{h}$. Using relations \eqref{eq:Doob1} and \eqref{eq:Doob2} it is indeed possible to write a condition such that the one-parameter group $T_a$ of the stochastic transformations is generated by an infinitesimal  stochastic transformation $V=(Y,C,\tau,H)$. Indeed if in equations \eqref{eq:Doob1} and \eqref{eq:Doob2} both $h_a$ and $\mathfrak{h}_a$ depend on a parameter $a$ and we take the derivative with respect to that parameter in $a=0$ (when $h_0=0$ as well as $\mathfrak{h}_0=0$), we obtain that there exists a function $k=\partial_a(\mathfrak{h}_a)|_{a=0}$ such that
\begin{eqnarray}
H_{\alpha}(x)&=&\sigma^i_{\alpha}(x)\partial_{x^i}(k)(x)\label{eq:Doob3}\\
0&=&L(k).\label{eq:Doob4}
\end{eqnarray}
Using the previous characterization we can state the following theorem.
\begin{theorem}\label{TeoDetEqDoob}
An infinitesimal  stochastic transformation $V=(Y,C,\tau,H)$ is a symmetry of the SDE $(\mu,\sigma)$ involving only Doob transformations with respect to $(\mu,\sigma)$ if and only if $V$ generates a one parameter group of transformations and  there exists a smooth function $k$ such that the following equations hold
\[
\begin{split}
H-\sigma^T \cdot \nabla k &=0\\
Y(\mu)-L(Y)-\sigma \cdot \sigma^T \cdot \nabla k &=0\\
\left[Y,\sigma \right] +\frac{1}{2}\tau \sigma + \sigma \cdot C&=0\\
L(k)&=0.
\end{split}
\]
\end{theorem}
\begin{proof}
The proof is an easy consequence of the determining equations in Theorem \ref{theorem_determining} and equations \eqref{eq:Doob3} and \eqref{eq:Doob4}.
\hfill\end{proof}

\subsection{Relation between determining equations for PDEs and SDEs}

In this section we compare the determining equations for the symmetries of an SDE of the form \eqref{eqTIMESDE} with the determining equations for Lie point infinitesimal symmetries of the corresponding Kolmogorov equation \eqref{KOLMOGOROV}. In particular, we prove the following result.
\begin{theorem}\label{daSDEaPDE}
  Let $V=(Y,C,\tau, H)$ be an infinitesimal symmetry of Doob type  for the SDE \eqref{eqTIMESDE}, with $Y=m(x,z)\frac {\partial }{\partial z} + \phi^i(x,z) \frac {\partial }{\partial x^i}$ and $H =\sigma^T\cdot \nabla k$. Then the vector field $
  \Xi=Y-k(x,z)u \frac {\partial }{\partial u}$ is a Lie point infinitesimal symmetry for the associated Kolmogorov equation  \eqref{KOLMOGOROV}.
\end{theorem}

\begin{proof}
 Since $V$ is a symmetry for \eqref{eqTIMESDE}, Theorem \ref{TeoDetEqDoob} provides the following determining equations for $V$
\begin{eqnarray}
  L(k) &=& 0  \label{eq:detSDE1} \\
  L(m) &=& \tau  \label{eq:detSDE2} \\
  L(\phi)-Y(\mu) +\sigma \cdot \sigma^T \cdot \nabla k-\tau \mu &=& 0  \label{eq:detSDE3} \\
  Y(\sigma)-\nabla \phi \cdot \sigma +\frac 12 \tau \sigma +\sigma \cdot C &=& 0 \label{eq:detSDE4} \\
  \sigma^T \cdot \nabla m &=& 0  \label{eq:detSDE5}
\end{eqnarray}
We want to prove that, if $A=\frac 12 \sigma \cdot \sigma^T$ is of constant rank, the vector field $\Xi=Y-ku\frac {\partial}{\partial u}$  satisfies the determining equation for the Kolmogorov PDE \eqref{KOLMOGOROV}.\\
The first condition, i.e. $L(k)=0$, is the same in both the sets of determining equations and, since $\mu^i$ do not depend on $u$, $\Xi(\mu^i)=Y(\mu^i)$ and equations \eqref{eq:detSDE2} and \eqref{eq:detSDE3} imply \eqref{eq:detPDE3}. Moreover, multiplying on the left \eqref{eq:detSDE5} by $\sigma$, we immediately get \eqref{eq:detPDE5}. In order to prove that \eqref{eq:detPDE4} holds, we consider the right multiplication of \eqref{eq:detSDE4} by $\sigma^T$ and we get, using $\sigma \cdot \sigma^T=2A$,
\begin{equation}\label{eq:detSDE4var}
Y(\sigma)\sigma^T- 2  \nabla \phi \cdot  A+  \tau A+\sigma \cdot  C \cdot  \sigma^T  =0
\end{equation}
Hence, considering the semi-sum of \eqref{eq:detSDE4var} and its transposed and using the fact that $C$ is an antisymmetric matrix, we get
\begin{equation*}
  Y(A)+\tau A-  \nabla \phi \cdot  A -A\cdot  (\nabla \phi)^T=0.
\end{equation*}
Since $A$ does not depend  on $u$, $Y(A)=\Xi(A)$ and, using \eqref{eq:detSDE2}, we get \eqref{eq:detPDE4}
\end{proof}

Now we prove the converse of the previous theorem, showing that with any infinitesimal Lie point symmetry of the Kolmogorov equation it is possible to associate an  infinitesimal Doob symmetry of the corresponding SDE.

\begin{theorem}\label{daPDEaSDE}
Let $\Xi=m(x,z)\frac {\partial }{\partial z} + \phi^i(x,z) \frac {\partial }{\partial x^i}-k(x,z)u\frac {\partial }{\partial u} $ be a  Lie point symmetry for the PDE  \eqref{KOLMOGOROV}, with $A=\frac 12 \sigma \cdot \sigma^T$ of constant rank. Then there exist smooth functions $C:\mathbb{R}^{n+1}\to \mathfrak{so}(m)$ and $\tau: \mathbb{R}^{n+1}\to \mathbb{R}$, such that $V=(Y,C,\tau, H)$, with  $Y=\Xi+k(x,z)u \frac {\partial }{\partial u}$  and $H=\sigma^T \cdot \nabla k$,  is a Doob infinitesimal symmetry for the SDE \eqref{eqTIMESDE}.

\end{theorem}

\begin{proof} As in the proof of Theorem \ref{daSDEaPDE}, \eqref{eq:detSDE1} and \eqref{eq:detPDE2} coincide and, in order to satisfy \eqref{eq:detSDE2} and \eqref{eq:detSDE3} we have to choose $\tau=L(m)$ and exploit the fact that $\Xi(\mu^i)=Y(\mu^i)$, since $\mu^i$ do not depend on $u$. Moreover,  the hypothesis of constant rank for $A$ ensures that  \eqref{eq:detPDE5} implies  \eqref{eq:detSDE5}. Hence, we have to choose the matrix $C$ in order to satisfy \eqref{eq:detSDE4}. Considering the left product of \eqref{eq:detSDE4} with $\sigma^T$, we get the following expression for $C$, where, once again, we exploit the fact that $A$ is of constant rank
\begin{equation*}
  C=(\sigma^T\cdot \sigma)^{-1}\sigma^T\nabla \phi \cdot \sigma -(\sigma^T\cdot \sigma)^{-1}\sigma^TY(\sigma)-\frac 12 \tau I.
\end{equation*}
The last step is to prove that the matrix $C$ defined above  is antisymmetric. Following the same line of the proof of Theorem \ref{daSDEaPDE}, if we consider  the sum of \eqref{eq:detSDE4var} with its transpose, we find
\begin{equation*}
Y(\sigma \cdot \sigma^T)+2\tau A -2 \nabla \phi \cdot  A -2A \cdot (\nabla\phi)^T +\sigma \cdot (C+C^T) \cdot \sigma^T=0
\end{equation*}
and, using \eqref{eq:detPDE4}, we get $\sigma \cdot  (C+C^T)\cdot \sigma^T=0$.  Therefore, since $A$ is of constant rank,  $C+C^T=0$, i.e. $C$ is an antisymmetric matrix, and this concludes the proof.
\end{proof}

\begin{remark}
Theorems \ref{daSDEaPDE} and \ref{daPDEaSDE} show that there exists a one-to-one correspondence between infinitesimal symmetries of Doob type of the SDE and Lie point infinitesimal symmetries of the corresponding Kolmogorov equation. On the other hand, if we allow more general changes of the probability measure of Girsanov type, the family of symmetries of the SDE introduced in this paper is  wider than the family  of the symmetries for the corresponding Kolmogorov equation. The new (non Doob type) symmetries can belong to two different families.  \\
The first family consists of non Doob symmetries such that the function $k$ satisfying $H=\sigma^T\nabla k$ does exist, but does not satisfies $L(k)=0$. We call these kind of symmetries \emph{almost Doob symmetries}. \\
The second family of non Doob symmetries contains  symmetries $V=(Y,C,\tau, H)$ such that does not exist any function $k$ for which  $H= \sigma^T\nabla k $. We remark that the second family is not empty only for SDEs driven by  $m$ dimensional  Brownian motions with $m>1$, while for SDEs driven by one-dimensional Brownian motions  only almost Doob symmetries exist. In the next section we provide some examples of this fact and, in particular, we find almost Doob symmetries for one-dimensional Brownian motion,  Ornstein-Uhlenbeck process and  CIR model. Moreover, we show that  two-dimensional Brownian motion admits non-Doob infinitesimal symmetries of both the types described above.
\end{remark}

\section{Examples \label{section_4}}

In this last section we propose some non-trivial examples, in order to show that  the introduction of the probability measure change allows us  to obtain a larger number of symmetries with respect to the methods available in the previous literature.

\subsection{One-dimensional Brownian motion}

Let us consider  the following SDE
\begin{equation*}
\binom{dX_t}{dZ_t}=\binom{0}{1}dt+\binom{1}{ 0 }  dW_t^1
\end{equation*}
corresponding to a one-dimensional Brownian motion, where the second component, admitting solution $Z_t=t$, has been introduced in order to include in our setting also time dependent transformations.
\par
If we are interested in finding  symmetries of Doob type for this SDE, we have to look for
\begin{equation}\label{GenSymmD1}
V=\left(Y,C,\tau, H\right)=\left(\binom{f\left(x,z\right)}{m\left(x,z\right)},0,\tau\left(x,z\right), H_1(x,z)\right)
\end{equation}
where $H_1=k_x$ and the functions $k,f,m,\tau$ satisfy the following determining equations
\begin{equation*}
\begin{split}
\frac{1}{2}k_{xx}+k_{z}&=0\\
\frac{1}{2}f_{xx}+f_{z}+k_x&=0\\
\frac{1}{2}m_{xx}+m_{z}&= \tau \\
f_{x}&=\frac{1}{2}\tau\\
m_{x}&=0.
\end{split}
\end{equation*}
In order to solve this system, we look for a function $k(x,z)$ satisfying the first equation.  If we chose $k=c_1+c_2x+c_3(x^2-z)$, we can solve the remaining equations and we find
 \begin{equation*}
\begin{split}
f(x,z)&=-c_2z-2c_3xz+c_4x+c_5\\
m(x,z)&= -2c_3z^2+2c_4z+c_6\\
\tau(x,z)&= -4c_3z+2c_4\\
\end{split}
\end{equation*}
 Therefore, the Doob symmetries for the one-dimensional Brownian motion associated with the function $k=c_1+c_2x+c_3(x^2-z)$ are

  \begin{equation*}
V_1=\left(\binom{z}{0}, 0,0, -1\right)
\end{equation*}

  \begin{equation*}
V_2=\left(\binom{2xz}{2z^2}, 0,4z, -2x\right)
\end{equation*}

  \begin{equation*}
V_3=\left(\binom{x}{2z},0,2, 0\right)
\end{equation*}

  \begin{equation*}
V_4=\left(\binom{1}{0}, 0,0, 0\right)
\end{equation*}

 \begin{equation*}
V_5=\left(\binom{0}{1}, 0,0, 0\right)
\end{equation*}

Hence, exploiting Theorem \ref{daSDEaPDE} and adding  the symmetry $V_6$ corresponding to  $H_1=0$, we recover, in our framework, the following generators of the well known six dimensional algebra of Lie point symmetries for the Kolmogorov equation associated with one dimensional Brownian motion
 \begin{equation*}
\begin{split}
\Xi_1&= z \frac {\partial}{\partial x}+xu\frac {\partial}{\partial u}\\
\Xi_2&= 2xz \frac {\partial}{\partial x}+ 2z^2 \frac {\partial}{\partial z}+ u(x^2-z)\frac {\partial}{\partial u}\\
\Xi_3&= x \frac {\partial}{\partial x}+2z\frac {\partial}{\partial z}\\
\Xi_4&= \frac {\partial}{\partial x}\\
\Xi_5&=  \frac {\partial}{\partial z}\\
\Xi_6&= u\frac {\partial}{\partial u}.\\
\end{split}
\end{equation*}

Let us now consider the determining equations for general symmetries (not necessarily of Doob type) of the one-dimensional Brownian motion. In this case the determining equations provided by Theorem \ref{theorem_determining} are
\begin{equation*}
\begin{split}
\frac{1}{2}f_{xx}+f_{z}+H_1&=0\\
\frac{1}{2}m_{xx}+m_{z}&= \tau \\
f_{x}&=\frac{1}{2}\tau\\
m_{x}&=0.\\
\end{split}
\end{equation*}
Here we have no conditions on $H_1$ and we can find  the following infinite-dimensional family of almost Doob symmetries (depending on an arbitrary deterministic function of the time $ \alpha(z)$)
  \begin{equation*}
\tilde V_{\alpha}=
\left( \left(
               \begin{array}{c}
                 \frac 12 \alpha(z) x \\
                 \int \alpha(z) \, dz\\
               \end{array}
             \right),
             0,\alpha(z),-\frac 12 x \alpha'(z)
              \right).
\end{equation*}
 We remark that, even if $H$ satisfies the condition $H=H_1=\sigma^T \nabla k$ for $k=-\frac 14 x^2\alpha'(z)$, we cannot ensure that
 \begin{equation*}
\frac 12 k_{xx}+ k_{z}=0
 \end{equation*}
 since this condition holds only if $\alpha'(z)=0$. Therefore, for general (non constant) functions $\alpha(z)$,
  $\tilde V_{\alpha}$ is not a Doob symmetry and Theorem \ref{daSDEaPDE} does  not apply. This means that we cannot associate with  $\tilde V_{\alpha}$ any symmetry of the  Kolmogorov PDE $\frac 12 u_{xx}+u_z=0$.

\begin{remark}\label{remark_timesymmetry}
The symmetries $V_{\alpha}$ are related with the \emph{random time invariance} of Brownian motion (see Section \ref{section_1}).  Indeed, if we consider the deterministic  time change $t'=f(t)$, then $W'_{t'}=\int_0^{t'}{\sqrt{f'(f^{-1}(s))}dW_{f^{-1}(s)}}$ is a new Brownian motion (see \cite{DMU}). The symmetry $\tilde V_{\alpha}$ provides a generalization of the previous fact,  considering $\tilde{W}_{t'}=\sqrt{f'(f^{-1}(t'))}W_{f^{-1}(t')}$ instead of the integral $\int_0^{t'}{\sqrt{f'(f^{-1}(s)}dW_{f^{-1}(s)}}$. Obviously, since $W'_{t'}$ is a $\mathbb{P}$-Brownian motion, $\tilde{W}_{t'}$ cannot be a $\mathbb{P}$-Brownian motion. But, since $\tilde V_{\alpha}$ is a symmetry, using Girsanov transformation we have that $\tilde{W}_{t'}$ is a $\mathbb{Q}$-Brownian motion, where $\mathbb{Q}$ has density with respect to $\mathbb{P}$ given by
\[\frac{\text{d}\mathbb{Q}_\mathcal{T}}{\text{d}\mathbb{P}_\mathcal{T}}=\exp\left(\int_0^{\mathcal{T}}{\frac{f''(s)}{2\sqrt{f(s)}}W_s dW_s}-\frac{1}{2}\int_0^\mathcal{T}{\left(\frac{f''(s)W_s}{2\sqrt{f(s)}}\right)^2ds} \right). \]
\end{remark}

\subsection{Ornstein-Uhlenbeck model}
Let us consider the Ornstein-Uhlenbeck process,  solution to the following SDE
\begin{equation*}
\binom{dX_t}{dZ_t}=\binom{aX_t+b}{1}dt+\binom{1 }{ 0  } dW_t^1
\end{equation*}
\\
where, as in the previous example, the second component has been introduced in order to include in our framework time dependent transformations. In this case, if we look for a Doob type symmetries of the form \eqref{GenSymmD1} and we consider $H_1=k_x$, the determining equations are

\begin{equation}
\begin{split}
\frac{1}{2}k_{xx}+\left(ax+b\right)k_{x}+k_z&=0\\
\frac{1}{2}f_{xx}+\left(ax+b\right)f_x+f_z&=af+ \tau \left(ax+b\right)-k_x\\
\frac{1}{2}m_{xx}+\left(ax+b\right)m_x+m_z&=\tau\\
f_x&=\frac{1}{2} \tau\\
m_x&=0.\\
\end{split}\label{deteqOU}
\end{equation}
With a suitable ansatz we can find a solution to the first equation of the form
\begin{equation*}
\begin{split}
k\left(x,z\right)&=c_1+c_2\left(ax+b\right)e^{-az}+c_3\left(ax^2+2bx+\dfrac{a+2b^2}{2a}\right)e^{-2az}\\
\end{split}
\end{equation*}
so that
\begin{equation*}
\begin{split}
H_1\left(x,z\right)&=c_2ae^{-az}+2c_3\left(ax+b\right)e^{-2az}.
\end{split}
\end{equation*}
The remaining determining equations read
\begin{equation}
\begin{split}
-\left(ax+b\right)f_x+f_z&=af-c_2ae^{-az}-2c_3(ax+b)e^{-2az}\\
m_z&=\tau\\
\tau&=2f_x\\
m_x&=0,
\end{split}
\end{equation}
whose solution is
\begin{equation*}
\begin{split}
f(x,z))&=\frac {c_3}{2a} e^{-2az}(ax+b)+\frac {c_4}{2a} e^{2az}(ax+b) +\frac {c_2}{2}e^{-az} +c_5e^{az}\\
m(x,z)&= -\frac {c_3}{2a} e^{-2az}+\frac {c_4}{2a}e^{2az} +c_6\\
\tau (x,z) &= c_3 e^{-2az}+c_4 e^{2az}.
\end{split}
\end{equation*}
Therefore we have the following  family of infinitesimal symmetries of Doob type for the Ornstein-Uhlenbeck SDE
\begin{equation*}
V_1=\left(\binom{\frac{1}{2}e^{-az}}{0}, 0,0, ae^{-az}\right)
\end{equation*}

\begin{equation*}
V_2=\left(\binom{\frac{1}{2a}e^{-2az} (ax+b)}{-\frac {1}{2a} e^{-2az}}, 0,e^{-2az}, 2(ax+b)e^{-2az}\right)
\end{equation*}

\begin{equation*}
V_3=\left(\binom{\frac{1}{2a}e^{2az}(ax+b)}{\frac {1}{2a} e^{2az}}, 0,e^{2az}, 0\right)
\end{equation*}

\begin{equation*}
V_4=\left(\binom{e^{az}}{0}, 0,0, 0\right)
\end{equation*}

\begin{equation*}
V_5=\left(\binom{0}{1}, 0,0, 0\right).
\end{equation*}

Using Theorem \ref{daSDEaPDE}, we can associate with $V_i$ the generators for the  symmetry algebra of the Kolmogorov equation associated with the Ornstein-Uhlenbeck SDE, where the vector field $\Xi_6$ corresponds to the trivial symmetry $V_6$ associated with $H_1=0$.
 \begin{equation*}
\begin{split}
\Xi_1&= \frac 12 e^{-az} \frac {\partial}{\partial x}-u(ax+b)e^{-az}\frac {\partial}{\partial u}\\
\Xi_2&= \frac {1}{2a} e^{-2az}(ax+b) \frac {\partial}{\partial x}-\frac {1}{2a} e^{-2az} \frac {\partial}{\partial z}- ue^{-2az} (ax^2+2bx+\frac {a+2b^2}{2a})\frac {\partial}{\partial u}\\
\Xi_3&= \frac {1}{2a} e^{2az}(ax+b) \frac {\partial}{\partial x}+ \frac {1}{2a} e^{2az}\frac {\partial}{\partial z}\\
\Xi_4&= e^{az}\frac {\partial}{\partial x}\\
\Xi_5&=  \frac {\partial}{\partial z}\\
\Xi_6&= u\frac {\partial}{\partial u}.\\
\end{split}
\end{equation*}

Let us now consider the determining equations for general symmetries (not necessarily of Doob type)  for the Ornstein-Uhlenbeck process. In this case the determining equations provided by Theorem \ref{theorem_determining} are
\begin{equation}
\begin{split}
\frac{1}{2}f_{xx}+\left(ax+b\right)f_x+f_z&=af+ \tau \left(ax+b\right)-H_1\\
\frac{1}{2}m_{xx}+\left(ax+b\right)m_x+m_z&=\tau\\
f_x&=\frac{1}{2} \tau\\
m_x&=0.\\
\end{split}\label{deteqOU}
\end{equation}
Solving these equations we find the following  symmetries
\begin{equation*}
\tilde V_1=\left(\binom{\frac{1}{2}(\frac 1a +e^{2az})x-\frac 12 \frac {b}{a^2}+\frac 12 \frac ba e^{2az} }{\frac z a + \frac {1}{2a} e^{2az}}, 0,\frac 1 a +e^{2az}, x \right)
\end{equation*}

\begin{equation*}
\tilde V_2=\left(\binom{\frac {x}{2a}-\frac {b}{2a^2}+e^{az}}{\frac za}, 0,\frac 1a, x\right).
\end{equation*}
We remark that $\tilde V_1$ and $\tilde V_2$ are not symmetries of Doob type, since we have $H_1=x=\sigma^T \nabla k$, for $k=\frac 12 x^2+ k_1(z)$, but $L(k)\not= 0$.

\subsection{CIR model}
In this section we look for the symmetries of the CIR model

\begin{equation*}
\binom{dX_t}{dZ_t}=\binom{aX_t+b}{1}dt+
\binom{\sigma_0\sqrt{X_t}}{0}
dW_t^1,
\end{equation*}
which is widely used in mathematical finance to describe the behavior of the interest rates (see \cite{Brigo2006} Chapter 3). We remark that CIR model is also important for the fact that closed formulas for its Laplace transform and its transition probability are known. In particular, the advantages of the knowledge of Lie point symmetries of the Kolmogorov equation associated with CIR model and  their relationship with closed formulas is  discussed in \cite{Craddock2004}.\\

If we consider  infinitesimal transformations $V$ of the form \eqref{GenSymmD1} and $H_1=\sigma_0\sqrt{x}k_x$, the determining equations for the symmetries of Doob type for the CIR model are
\begin{equation*}
\begin{split}
\frac{1}{2}\sigma_0^2 x k_{xx}+\left(ax+b\right)k_x+k_z&=0\\
\frac{1}{2}\sigma_0^2 x f_{xx}+\left(ax+b\right)f_x+f_z&=af+(ax+b)\tau -\sigma^2_0x k_x  \\
\frac{1}{2}\sigma_0^2 x m_{xx}+\left(ax+b\right)m_x+m_z&=\tau\\
\sigma_0\sqrt{x}f_x-\frac{1}{2}\sigma_0\frac{1}{\sqrt{x}}f&=\frac{1}{2}\sigma_0\sqrt{x} \tau\\
\sigma_0\sqrt{x} \ m_x&=0.\\
\end{split}
\end{equation*}
The first step to solve these equations is to find a  function $k$ satisfying the first determining equation. If we chose
\begin{equation*}
\begin{split}
k\left(x,z\right)&=c_1\left(ax+b\right)e^{-az}+c_2\\
\end{split}
\end{equation*}
we can solve the remaining determining equations and we obtain
\begin{equation*}
\begin{split}
f(x,z)&=\frac 12 c_1\sigma_0^2xe^{-az}+c_3xe^{az}\\
m(x,z)&=-\frac {1}{2a}c_1\sigma_0^2e^{-az}+\frac {c_3}{a}e^{az}+c_4 \\
\tau (x,z) &=\frac{ 1}{2}c_1\sigma_0^2 e^{-az}+c_3 e^{az}
\end{split}
\end{equation*}
corresponding to the infinitesimal symmetries
\begin{equation*}
V_1=\left(\binom{\frac{\sigma_0^2}{2}xe^{-az}}{-\frac{\sigma_0^2}{2a}e^{-az}},0,\frac{\sigma_0^2}{2}e^{-az},\sigma_0a\sqrt{x}e^{-az}\right)
\end{equation*}
\begin{equation*}
V_2=\left(\binom{xe^{az}}{\frac 1a e^{az}},0,e^{az},0\right)
\end{equation*}
\begin{equation*}
V_3=\left(\binom{0}{1},0,0,0\right).
\end{equation*}
Once again we can exploit Theorem \ref{daSDEaPDE} in order to find the Lie point infinitesimal symmetries for the Kolmogorov equation associated with the  CIR model, i.e.
 \begin{equation*}
\begin{split}
\Xi_1&= \frac{\sigma_0^2}{2}xe^{-az}\frac {\partial}{\partial x}-\frac{\sigma_0^2}{2a}e^{-az}\frac {\partial}{\partial z} -u(ax+b)e^{-az}\frac {\partial}{\partial u}\\
\Xi_2&= xe^{az} \frac {\partial}{\partial x}+\frac 1a e^{az}\frac {\partial}{\partial z}\\
\Xi_3&=\frac {\partial}{\partial z}\\
\Xi_4&= u\frac {\partial}{\partial u}
\end{split}
\end{equation*}
where the last symmetry, as usual, corresponds  to the choice $H_1=0$. They form the generators of all non trivial symmetries of Kolmgorov CIR model equation (see \cite{Craddock2004}).\\

Let us now consider the determining equations for general symmetries (not necessarily of Doob type)  for the CIR model.
In this case the determining equations provided by Theorem \ref{theorem_determining} are
\begin{equation*}
\begin{split}
\frac{1}{2}\sigma_0^2 x f_{xx}+\left(ax+b\right)f_x+f_z&=af+(ax+b)\tau -\sigma_0\sqrt x H_1  \\
\frac{1}{2}\sigma_0^2 x m_{xx}+\left(ax+b\right)m_x+m_z&=\tau\\
\sigma_0\sqrt{x}f_x-\frac{1}{2}\sigma_0\frac{1}{\sqrt{x}}f&=\frac{1}{2}\sigma_0\sqrt{x} \tau\\
\sigma_0\sqrt{x} \ m_x&=0,\\
\end{split}
\end{equation*}
and we can find the solution
\begin{equation*}
\begin{split}
f(x,z)&=c_1\sqrt{x}+c_2 \frac {\sigma_0}{a} x \\
m(x,z)&=c_2 \frac {\sigma_0}{a}z\\
\tau (x,z) &=c_2 \frac{ \sigma_0}{a}\\
H_1&=c_1\left(\frac{\sigma_0}{8x}-\frac{b}{2\sigma_0 x}+\frac{a}{2\sigma_0 } \right)+ c_2 \sqrt x\\
\end{split}
\end{equation*}
corresponding to the infinitesimal symmetries
\begin{equation*}
\tilde V_1=\left(\binom{\sqrt{x}}{0},0,0,\frac{\sigma_0}{8x}-\frac{b}{2\sigma_0 x}+\frac{a}{2\sigma_0}\right)
\end{equation*}
\begin{equation*}
\tilde V_2=\left(\binom{\frac{\sigma_0}{a}x }{\frac{\sigma_0}{a}z},0,\frac{\sigma_0}{a},\sqrt x\right).
\end{equation*}
We remark that $\tilde V_1$ and $\tilde V_2$ are not symmetries of Doob type, since we have $H_1=\sqrt x=\sigma^T \nabla k_i$, for $k_1=-\frac{1}{4\sqrt{x}}+\frac{b}{\sigma_0^2\sqrt{x}}+\frac{a\sqrt x}{\sigma_0^2}$ and $k_2=\frac { x}{\sigma_0} $ but, as in the previous cases, $k_i$ do not satisfy $L(k_i)= 0$.

\subsection{Two-dimensional Brownian motion}

Let us consider  the following SDE
\begin{equation*}
  \left(
  \begin{array}{c}
    dX_t \\
    dY_t \\
    dZ_t\\
  \end{array}
\right)=
\left(
  \begin{array}{c}
    0 \\
    0 \\
    1\\
  \end{array}
\right) dt + \left(
               \begin{array}{cc}
                 1 & 0 \\
                 0 & 1 \\
                 0 & 0\\
               \end{array}
             \right)
                \left(
  \begin{array}{c}
    dW^1_t \\
    dW^2_t \\
  \end{array}
\right)
\end{equation*}
corresponding to a two-dimensional Brownian motion. As in the previous examples, the third component $dZ_t=dt$ has been introduced in order to include time dependent transformations in our setting.
\par
When we look for Doob type infinitesimal symmetries for this SDE, we have to find an infinitesimal transformation $V=\left(Y,C,\tau, H\right)$
\begin{equation*}
V=
\left( \left(
               \begin{array}{c}
                 f(x,y,z) \\
              g(x,y,z)\\
                 m(x,y,z)\\
               \end{array}
             \right),
             \left(
               \begin{array}{cc}
                 0 & c(x,y,z) \\
                 -c(x,y,z) & 0\\
                               \end{array}
             \right),\tau\left(x,y,z\right),
             \left(
  \begin{array}{c}
    H_1(x,y,z)\\
    H_2(x,y,z) \\
  \end{array}
\right) \right)
\end{equation*}
where $H_1=k_x, H_2=k_y$ and the functions $k,f,g,m,\tau, c$ satisfy the following determining equations
\begin{equation*}
\begin{split}
\frac{1}{2}k_{xx}+\frac{1}{2}k_{yy}+k_{z}&=0\\
\frac{1}{2}f_{xx}+\frac{1}{2}f_{yy}+f_{z}+k_x&=0\\
\frac{1}{2}g_{xx}+\frac{1}{2}g_{yy}+g_{z}+k_y&=0\\
\frac{1}{2}m_{xx}+\frac{1}{2}m_{yy}+ m_{z}&= \tau \\
f_{x}&=\frac{1}{2}\tau\\
f_y&=c\\
g_x&=-c\\
g_y&=\frac 12 \tau\\
m_{x}&=0\\
m_y&=0.
\end{split}
\end{equation*}
 If we chose $k=c_1+c_2x+c_3y+ c_4(x^2+y^2-2z)$, we can solve the previous system and we find
 \begin{equation*}
\begin{split}
f(x,y,z)&=-2c_4xz+\frac 12 c_5x-c_2z+c_6 y +c_7\\
g(x,y,z)&= -2c_4yz+\frac 12 c_5y-c_3z-c_6 x +c_8\\
m(x,y,z)&= -2c_4z^2+c_5z+c_9\\
\tau(x,y,z)&= -4c_4z+c_5\\
c(x,y,z)&=c_6.
\end{split}
\end{equation*}
 Therefore, the infinitesimal symmetries for the two-dimensional Brownian motion associated with the  Doob function $k=c_1+c_2x+c_3y+ c_4(x^2+y^2-2z)$ are
\begin{equation*}
V_1=
\left( \left(
               \begin{array}{c}
                 -z \\
              0\\
                0\\
               \end{array}
             \right),
             \left(
               \begin{array}{cc}
                 0 & 0 \\
                 0 & 0\\
                               \end{array}
             \right),0,
             \left(
  \begin{array}{c}
    1\\
    0 \\
  \end{array}
\right) \right)
\end{equation*}

\begin{equation*}
V_2=
\left( \left(
               \begin{array}{c}
                 0 \\
              -z\\
                0\\
               \end{array}
             \right),
             \left(
               \begin{array}{cc}
                 0 & 0 \\
                 0 & 0\\
                               \end{array}
             \right),0,
             \left(
  \begin{array}{c}
    0\\
    1 \\
  \end{array}
\right) \right)
\end{equation*}

\begin{equation*}
V_3=
\left( \left(
               \begin{array}{c}
                 -2xz \\
              -2yz\\
                -2z^2\\
               \end{array}
             \right),
             \left(
               \begin{array}{cc}
                 0 & 0 \\
                 0 & 0\\
                               \end{array}
             \right),-4z,
             \left(
  \begin{array}{c}
    2x\\
    2y \\
  \end{array}
\right) \right)
\end{equation*}

\begin{equation*}
V_4=
\left( \left(
               \begin{array}{c}
                 \frac 12 x \\
              \frac 12 y\\
                z\\
               \end{array}
             \right),
             \left(
               \begin{array}{cc}
                 0 & 0 \\
                 0 & 0\\
                               \end{array}
             \right),1,
             \left(
  \begin{array}{c}
    0\\
    0 \\
  \end{array}
\right) \right)
\end{equation*}

\begin{equation*}
V_5=
\left( \left(
               \begin{array}{c}
                 y \\
              -x\\
                0\\
               \end{array}
             \right),
             \left(
               \begin{array}{cc}
                 0 & 1 \\
                 -1 & 0\\
                               \end{array}
             \right),0,
             \left(
  \begin{array}{c}
    0\\
    0 \\
  \end{array}
\right) \right)
\end{equation*}

\begin{equation*}
V_6=
\left( \left(
               \begin{array}{c}
                 1 \\
              0\\
                0\\
               \end{array}
             \right),
             \left(
               \begin{array}{cc}
                 0 & 0 \\
                 0 & 0\\
                               \end{array}
             \right),0,
             \left(
  \begin{array}{c}
    0\\
    0 \\
  \end{array}
\right) \right)
\end{equation*}

\begin{equation*}
V_7=
\left( \left(
               \begin{array}{c}
                 0 \\
              1\\
                0\\
               \end{array}
             \right),
             \left(
               \begin{array}{cc}
                 0 & 0 \\
                 0 & 0\\
                               \end{array}
             \right),0,
             \left(
  \begin{array}{c}
    0\\
    0 \\
  \end{array}
\right) \right)
\end{equation*}

\begin{equation*}
V_8=
\left( \left(
               \begin{array}{c}
                 0 \\
              0\\
                1\\
               \end{array}
             \right),
             \left(
               \begin{array}{cc}
                 0 & 0 \\
                 0 & 0\\
                               \end{array}
             \right),0,
             \left(
  \begin{array}{c}
    0\\
    0 \\
  \end{array}
\right) \right).
\end{equation*}
Using Theorem \ref{daSDEaPDE}, we can associate with any $V_i$ a corresponding symmetry of the Kolmogorov equation
\begin{equation}\label{Kolm2dim}
\frac 12 u_{xx}+\frac 12 u_{yy}+ u_{z}=0.
\end{equation}
In this way we recover, adding  the trivial symmetry $V_9$ corresponding to  $H_1= H_2=0$, the following generators for the well known nine-dimensional Lie algebra of Lie point symmetries of \eqref{Kolm2dim}
 \begin{equation*}
\begin{split}
\Xi_1&= z \frac {\partial}{\partial x}+xu\frac {\partial}{\partial u}\\
\Xi_2&= z \frac {\partial}{\partial y}+yu\frac {\partial}{\partial u}\\
\Xi_3&= 2xz \frac {\partial}{\partial x}+  2yz \frac {\partial}{\partial y}+2z^2 \frac {\partial}{\partial z}+ u(x^2+y^2-2z)\frac {\partial}{\partial u}\\
\Xi_4&=\frac 12 x \frac {\partial}{\partial x}+\frac 12 y \frac {\partial}{\partial y}+z\frac {\partial}{\partial z}\\
\Xi_5&= y\frac {\partial}{\partial x}-x \frac {\partial}{\partial y}\\
\Xi_6&= \frac {\partial}{\partial x}\\
\Xi_7&=  \frac {\partial}{\partial y}\\
\Xi_8&=  \frac {\partial}{\partial z}\\
\Xi_9&= u\frac {\partial}{\partial u}.\\
\end{split}
\end{equation*}
Let us now consider the determining equations for general symmetries (not necessarily of Doob type) of the two-dimensional Brownian motion. In this case the determining equations provided by Theorem \ref{theorem_determining} are
\begin{equation*}
\begin{split}
\frac{1}{2}f_{xx}+\frac{1}{2}f_{yy}+f_{z}+H_1&=0\\
\frac{1}{2}g_{xx}+\frac{1}{2}g_{yy}+g_{z}+H_2&=0\\
\frac{1}{2}m_{xx}+\frac{1}{2}m_{yy}+m_{z}&= \tau \\
f_{x}&=\frac{1}{2}\tau\\
f_y&=c\\
g_x&=-c\\
g_y&=\frac 12 \tau\\
m_{x}&=0\\
m_y&=0.
\end{split}
\end{equation*}
Here we have no conditions on the components of the vector field $H$ and we can find two infinite-dimensional  families of symmetries (each one depending on an arbitrary function of $z$) which are not of Doob type.

 The first family is the generalization of the almost Doob symmetries already found for the one-dimensional Brownian motion, i.e.
  \begin{equation*}
\tilde V_{\alpha}=
\left( \left(
               \begin{array}{c}
                 \frac 12 \alpha(z) x \\
               \frac 12 \alpha(z) y\\
                \int \alpha(z) \, dz\\
               \end{array}
             \right),
             \left(
               \begin{array}{cc}
                 0 & 0 \\
                 0 & 0\\
                               \end{array}
             \right),\alpha(z),
             \left(
  \begin{array}{c}
    -\frac 12 x \alpha'(z)\\
     -\frac 12 y \alpha'(z) \\
  \end{array}
\right) \right).
\end{equation*}

 We remark that, despite  $H$ satisfies the condition $H=\sigma^T \nabla k$ for $k=-\frac 14 \alpha'(z) (x^2+y^2)$, we cannot ensure that
 \begin{equation*}
\frac 12 k_{xx}+\frac 12 k_{yy}+ k_{z}=0
 \end{equation*}
 since this condition is satisfied only if $\alpha'(z)=0$. Therefore, since for general (non constant) functions $\alpha(z)$,
  $\tilde V_{\alpha}$ is not a Doob symmetry, Theorem \ref{daSDEaPDE} does  not apply and we cannot associate with  $\tilde V_{\alpha}$ any symmetries of the  Kolmogorov PDE \eqref{Kolm2dim}.

 The second family of non Doob-type  symmetries is given by
  \begin{equation*}
\tilde V_{\beta}=
\left( \left(
               \begin{array}{c}
                 \beta(z) y \\
              - \beta(z) x\\
                0\\
               \end{array}
             \right),
             \left(
               \begin{array}{cc}
                 0 & \beta(z) \\

                 -\beta(z) & 0\\
                               \end{array}
             \right),0,
             \left(
  \begin{array}{c}
    -y\beta'(z)\\
     x\beta'(z) \\
  \end{array}
\right) \right)
\end{equation*}
and does not have an analogous for the one-dimensional Brownian motion. In fact, in this case the vector field  $H$ does not satisfy the condition $H=\sigma^T \nabla k= \nabla k$ for any $k$.

\begin{remark}
The presence of the symmetries $\tilde V_{\beta}$ is related to the invariance with respect to random rotations of Brownian motion (see Section \ref{section_1}) in the same way as $\tilde V_{\alpha}$ is related to the invariance of Brownian motion with respect to random time changes (see Remark \ref{remark_timesymmetry}). Indeed if $B(t)$ is a smooth deterministic function  of time taking values in $O(2)$ we have that $\tilde{W}_t=B(t) \cdot W_t$ is a Brownian motion with respect to the new probability measure $\mathbb{Q}$ having density
\[\frac{\text{d}\mathbb{Q}_\mathcal{T}}{\text{d}\mathbb{P}_\mathcal{T}}=\exp \left(-\int_0^\mathcal{T}{W_s^T \cdot B'(s)^T \cdot dW_s}-\frac{1}{2} \int_0^\mathcal{T}{ |B'(s) \cdot W_s |^2 ds}\right).\]

\end{remark}

Therefore, the  introduction of the notion of stochastic transformation for an SDE allows us to enlarge the class of symmetries, showing that the stochastic approach is successful in providing  a wider family of symmetries which would be lost in a purely deterministic approach throughout the Kolmogorov equation associated with the SDE.

\section*{Acknowledgments}

We thank Giorgio Innocenti and Francesca De Masi for the discussions on the topic and the help with computations. The first author is supported by the German Research Foundation (DFG) via SFB 1060.

\bibliographystyle{plain}

\bibliography{doob2}

\begin{thebibliography}{10}

\bibitem{AlDeMoUg2018}
Sergio Albeverio, Francesco~C. De~Vecchi, Paola Morando, and Stefania Ugolini.
\newblock Random transformations and invariance of semimartingales on lie
  groups.
\newblock {\em arXiv preprint arXiv:1812.11066}, 2018.

\bibitem{AlDeMoUg2019}
Sergio Albeverio, Francesco~C. De~Vecchi, Paola Morando, and Stefania Ugolini.
\newblock {W}eak symmetries of stochastic differential equations driven by
  semimartingales with jumps.
\newblock {\em arXiv preprint arXiv:1904.10963}, 2019.

\bibitem{AlFei1995}
Sergio Albeverio and Shao-Ming Fei.
\newblock A remark on symmetry of stochastic dynamical systems and their
  conserved quantities.
\newblock {\em J. Phys. A}, 28(22):6363--6371, 1995.

\bibitem{Bluman}
George~W. Bluman and Sukeyuki Kumei.
\newblock {\em Symmetries and differential equations}, volume~81 of {\em
  Applied Mathematical Sciences}.
\newblock Springer-Verlag, New York, 1989.

\bibitem{Brigo2006}
Damiano Brigo and Fabio Mercurio.
\newblock {\em Interest rate models---theory and practice}.
\newblock Springer Finance. Springer-Verlag, Berlin, second edition, 2006.
\newblock With smile, inflation and credit.

\bibitem{doob}
Rapha\"{e}l Chetrite and Hugo Touchette.
\newblock Nonequilibrium {M}arkov processes conditioned on large deviations.
\newblock {\em Ann. Henri Poincar\'{e}}, 16(9):2005--2057, 2015.

\bibitem{Co1991}
M.~Cohen~de Lara.
\newblock A note on the symmetry group and perturbation algebra of a parabolic
  partial differential equation.
\newblock {\em J. Math. Phys.}, 32(6):1445--1449, 1991.

\bibitem{DeLara1995}
M.~Cohen~de Lara.
\newblock Geometric and symmetry properties of a nondegenerate diffusion
  process.
\newblock {\em Ann. Probab.}, 23(4):1557--1604, 1995.

\bibitem{Craddock2007}
Mark Craddock and Kelly~A. Lennox.
\newblock Lie group symmetries as integral transforms of fundamental solutions.
\newblock {\em J. Differential Equations}, 232(2):652--674, 2007.

\bibitem{Craddock2009}
Mark Craddock and Kelly~A. Lennox.
\newblock The calculation of expectations for classes of diffusion processes by
  {L}ie symmetry methods.
\newblock {\em Ann. Appl. Probab.}, 19(1):127--157, 2009.

\bibitem{Craddock2004}
Mark Craddock and Eckhard Platen.
\newblock Symmetry group methods for fundamental solutions.
\newblock {\em J. Differential Equations}, 207(2):285--302, 2004.

\bibitem{De2017}
Francesco~C. De~Vecchi.
\newblock Finite dimensional solutions to {SPDE}s and the geometry of infinite
  jet bundles.
\newblock {\em arXiv preprint arXiv:1712.08490}, 2017.

\bibitem{DeMo2016}
Francesco~C. De~Vecchi and Paola Morando.
\newblock The geometry of differential constraints for a class of evolution
  {PDE}s.
\newblock {\em arXiv preprint arXiv:1607.08014}, 2016.

\bibitem{DMU2}
Francesco~C. De~Vecchi, Paola Morando, and Stefania Ugolini.
\newblock Reduction and reconstruction of stochastic differential equations via
  symmetries.
\newblock {\em J. Math. Phys.}, 57(12):123508, 22, 2016.

\bibitem{DMU}
Francesco~C. De~Vecchi, Paola Morando, and Stefania Ugolini.
\newblock Symmetries of stochastic differential equations: a geometric
  approach.
\newblock {\em J. Math. Phys.}, 57(6):063504, 17, 2016.

\bibitem{DeMoUg2019}
Francesco~C. De~Vecchi, Paola Morando, and Stefania Ugolini.
\newblock A note on symmetries of diffusions within a martingale problem
  approach.
\newblock {\em Stoch. Dyn.}, 19(2):1950011, 21, 2019.

\bibitem{DeUg2017}
Francesco~C. De~Vecchi and Stefania Ugolini.
\newblock A symmetry-adapted numerical scheme for {SDE}s.
\newblock {\em arXiv preprint arXiv:1704.04167}, 2017.

\bibitem{dooblibro}
J.~L. Doob.
\newblock {\em Classical potential theory and its probabilistic counterpart},
  volume 262 of {\em Grundlehren der Mathematischen Wissenschaften [Fundamental
  Principles of Mathematical Sciences]}.
\newblock Springer-Verlag, New York, 1984.

\bibitem{ElJaLi2010}
K.~David Elworthy, Yves Le~Jan, and Xue-Mei Li.
\newblock {\em The geometry of filtering}.
\newblock Frontiers in Mathematics. Birkh\"{a}user Verlag, Basel, 2010.

\bibitem{Fredericks2007}
E.~Fredericks and F.~M. Mahomed.
\newblock Symmetries of first-order stochastic ordinary differential equations
  revisited.
\newblock {\em Math. Methods Appl. Sci.}, 30(16):2013--2025, 2007.

\bibitem{Ga2019}
G.~Gaeta.
\newblock W-symmetries of {I}to stochastic differential equations.
\newblock {\em J. Math. Phys.}, 60(5):053501, 29, 2019.

\bibitem{GaLui2017}
G.~Gaeta and C.~Lunini.
\newblock On {L}ie-point symmetries for {I}to stochastic differential
  equations.
\newblock {\em J. Nonlinear Math. Phys.}, 24(suppl. 1):90--102, 2017.

\bibitem{Ga2000}
Giuseppe Gaeta.
\newblock Lie-point symmetries and stochastic differential equations. {II}.
\newblock {\em J. Phys. A}, 33(27):4883--4902, 2000.

\bibitem{GaQui1999}
Giuseppe Gaeta and Niurka~Rodr\'{\i}guez Quintero.
\newblock Lie-point symmetries and stochastic differential equations.
\newblock {\em J. Phys. A}, 32(48):8485--8505, 1999.

\bibitem{GaSpa2017}
Giuseppe Gaeta and Francesco Spadaro.
\newblock Random {L}ie-point symmetries of stochastic differential equations.
\newblock {\em J. Math. Phys.}, 58(5):053503, 20, 2017.

\bibitem{Glo1991}
Joseph Glover.
\newblock Symmetry groups and translation invariant representations of {M}arkov
  processes.
\newblock {\em Ann. Probab.}, 19(2):562--586, 1991.

\bibitem{Glo1990}
Joseph Glover and Joanna Mitro.
\newblock Symmetries and functions of {M}arkov processes.
\newblock {\em Ann. Probab.}, 18(2):655--668, 1990.

\bibitem{Gungor20181}
F.~G\"{u}ng\"{o}r.
\newblock Equivalence and symmetries for variable coefficient linear heat type
  equations. {I}.
\newblock {\em J. Math. Phys.}, 59(5):051507, 31, 2018.

\bibitem{Gungor20182}
F.~G\"{u}ng\"{o}r.
\newblock Equivalence and symmetries for variable coefficient linear heat type
  equations. {II}. {F}undamental solutions.
\newblock {\em J. Math. Phys.}, 59(6):061507, 19, 2018.

\bibitem{Ko2010}
Roman Kozlov.
\newblock Symmetries of systems of stochastic differential equations with
  diffusion matrices of full rank.
\newblock {\em J. Phys. A}, 43(24):245201, 16, 2010.

\bibitem{Ko2012}
Roman Kozlov.
\newblock On symmetries of stochastic differential equations.
\newblock {\em Commun. Nonlinear Sci. Numer. Simul.}, 17(12):4947--4951, 2012.

\bibitem{Kozlov2018}
Roman Kozlov.
\newblock Random {L}ie symmetries of {I}t\^{o} stochastic differential
  equations.
\newblock {\em J. Phys. A}, 51(30):305203, 22, 2018.

\bibitem{Ko2019}
Roman Kozlov.
\newblock Symmetries of {I}t\^{o} stochastic differential equations and their
  applications.
\newblock In {\em Nonlinear systems and their remarkable mathematical
  structures. {V}ol. 1}, pages 408--436. CRC Press, Boca Raton, FL, 2019.

\bibitem{LaOr2008}
Joan-Andreu L\'{a}zaro-Cam\'{\i} and Juan-Pablo Ortega.
\newblock Stochastic {H}amiltonian dynamical systems.
\newblock {\em Rep. Math. Phys.}, 61(1):65--122, 2008.

\bibitem{LaOr2009}
Joan-Andreu L\'{a}zaro-Cam\'{\i} and Juan-Pablo Ortega.
\newblock Reduction, reconstruction, and skew-product decomposition of
  symmetric stochastic differential equations.
\newblock {\em Stoch. Dyn.}, 9(1):1--46, 2009.

\bibitem{Liao1992}
Ming Liao.
\newblock Symmetry groups of {M}arkov processes.
\newblock {\em Ann. Probab.}, 20(2):563--578, 1992.

\bibitem{Oksendal}
Bernt {\O}ksendal.
\newblock {\em Stochastic differential equations}.
\newblock Universitext. Springer-Verlag, Berlin, sixth edition, 2003.
\newblock An introduction with applications.

\bibitem{Olver}
Peter~J. Olver.
\newblock {\em Applications of {L}ie groups to differential equations}, volume
  107 of {\em Graduate Texts in Mathematics}.
\newblock Springer-Verlag, New York, second edition, 1993.

\bibitem{PriZa2010}
Nicolas Privault and Jean-Claude Zambrini.
\newblock Stochastic deformation of integrable dynamical systems and random
  time symmetry.
\newblock {\em J. Math. Phys.}, 51(8):082104, 19, 2010.

\bibitem{ReYo1999}
Daniel Revuz and Marc Yor.
\newblock {\em Continuous martingales and {B}rownian motion}, volume 293 of
  {\em Grundlehren der Mathematischen Wissenschaften [Fundamental Principles of
  Mathematical Sciences]}.
\newblock Springer-Verlag, Berlin, third edition, 1999.

\bibitem{RoWi2000}
L.~C.~G. Rogers and David Williams.
\newblock {\em Diffusions, {M}arkov processes, and martingales. {V}ol. 2}.
\newblock Cambridge Mathematical Library. Cambridge University Press,
  Cambridge, 2000.
\newblock It\^{o} calculus, Reprint of the second (1994) edition.

\bibitem{Srihirun2006}
B.~Srihirun, S.~V. Meleshko, and E.~Schulz.
\newblock On the definition of an admitted {L}ie group for stochastic
  differential equations with multi-{B}rownian motion.
\newblock {\em J. Phys. A}, 39(45):13951--13966, 2006.

\bibitem{ThiZa1997}
M.~Thieullen and J.~C. Zambrini.
\newblock Symmetries in the stochastic calculus of variations.
\newblock {\em Probab. Theory Related Fields}, 107(3):401--427, 1997.

\bibitem{Unal2003}
Gazanfer \"{U}nal.
\newblock Symmetries of {I}t\^{o} and {S}tratonovich dynamical systems and
  their conserved quantities.
\newblock {\em Nonlinear Dynam.}, 32(4):417--426, 2003.

\end{thebibliography}

\end{document}